%% file: main.tex
\documentclass[11pt,a4paper]{amsart}

\usepackage{amsmath, amsthm, amsfonts, amssymb}
\usepackage{mathtools}           
\usepackage{stmaryrd}           
\usepackage{MnSymbol}          

\usepackage{graphicx}           
\usepackage{tikz}
\usetikzlibrary{
  arrows.meta,
  positioning,
  decorations.pathmorphing,
  calc,
  shapes,
  fit,
  backgrounds,
  shadings,
  shadows
}
\usepackage{tikz-cd}

\usepackage{array}              
\usepackage[shortlabels]{enumitem}     
\usepackage{tasks}              

\usepackage{geometry}           
\usepackage{lscape}

\usepackage{framed}

\usepackage{xcolor}
\usepackage{hyperref}           

\usepackage[titletoc]{appendix}

\usepackage{subfiles}

\definecolor{darkred}{rgb}{0.8,0,0}
\definecolor{darkgreen}{rgb}{0,0.6,0}
\definecolor{darkblue}{rgb}{0,0,0.8}

\hypersetup{
    colorlinks,
    linkcolor=darkblue,
    filecolor=darkgreen,
    urlcolor=darkred,
    citecolor=darkgreen
}

\usetikzlibrary{arrows.meta,decorations.pathreplacing}
\let\oldmarginpar\marginpar
\renewcommand\marginpar[1]{\oldmarginpar{\tiny\bfseries\flushleft #1}}

\theoremstyle{plain}
\newtheorem{thm}{Theorem}[section]
\newtheorem{cor}[thm]{Corollary}
\newtheorem{lem}[thm]{Lemma}
\newtheorem{prop}[thm]{Proposition}

\newtheorem{ex}{Example}[section]

\theoremstyle{definition}
\newtheorem{defin}[thm]{Definition}

\theoremstyle{remark}

\newtheorem{rmk}[thm]{Remark}

\numberwithin{equation}{section}

\DeclareMathOperator{\Ext}{Ext}

\DeclareMathOperator{\Hom}{Hom}

\DeclareMathOperator{\Id}{Id}
\DeclareMathOperator{\Pic}{Pic}

\DeclareMathOperator{\Ad}{Ad}
\DeclareMathOperator{\Aut}{Aut}

\DeclareMathOperator{\End}{End}

\DeclareMathOperator{\vol}{vol}

\DeclareMathOperator{\tr}{tr}

\DeclareMathOperator{\Sym}{Sym}

\DeclareMathOperator{\supp}{Supp}
       
\DeclareMathOperator{\Nilp}{Nilp}

\newcommand{\st}{\ \middle|\ }

\DeclarePairedDelimiterX{\inn}[2]{\langle}{\rangle}{#1, #2}

\newcommand{\lie}[1]{\mathfrak{#1}}

\newcommand{\cA}{{\mathcal A}}

\newcommand{\cE}{{\mathcal E}}

\newcommand{\cG}{{\mathcal G}}

\newcommand{\cM}{{\mathcal M}}

\newcommand{\cO}{{\mathcal O}}
\newcommand{\cP}{{\mathcal P}}

\newcommand{\cU}{{\mathcal U}}
\newcommand{\cV}{{\mathcal V}}

\newcommand{\cY}{{\mathcal Y}}

\newcommand{\A}{\mathbb{A}} 
\newcommand{\R}{\mathbb{R}} 
\newcommand{\C}{\mathbb{C}} 
\newcommand{\K}{\mathbb{H}}
\newcommand{\Z}{\mathbb{Z}} 
 
\newcommand{\V}{\mathbb{V}}
\newcommand{\W}{\mathbb{W}} 

\newcommand{\Sp}{\mathrm{Sp}}

\newcommand{\SL}{\mathrm{SL}}

\title[Gaiotto loci and the nilpotent cone]{Gaiotto Loci and the Nilpotent Cone for $\Sp_{2n}(\C)$}
\author{Lucas Branco}
\email{lucasmpcastello@gmail.com}
\address{PUC-Rio, Departamento de Matem\'atica, Rio de Janeiro, Brazil}
\date{}
\subjclass[2020]{14H60, 14D20, 53D30, 14M15}
\keywords{Higgs bundles, Gaiotto loci, nilpotent cone, symplectic bundles, Brill--Noether loci}

\begin{document}

\begin{abstract}
Fix a theta characteristic on a compact Riemann surface and let $G$ be a connected complex semisimple Lie group equipped with a symplectic representation. The moment map sends a nonzero spinor with values in the associated representation bundle to a $G$-Higgs field, and the Zariski closure of the stable Higgs bundles obtained in this way is the corresponding Gaiotto locus. For an arbitrary symplectic representation, the Gaiotto locus is isotropic, and we give a Petri-type criterion for it to be Lagrangian. For the standard representation of $\Sp_{2n}(\C)$, with $n\geq 2$, where the moment map is $\psi\mapsto\psi\otimes\psi$, the Gaiotto locus lies in the nilpotent cone. We prove that it is the irreducible component obtained as the Bia\l ynicki--Birula closure associated with $\cU(\Sp_{2n-2}(\C))$. Its intersection with the stable cotangent chart is the closure of the conormal bundle to the one-spinor stratum of the generalized theta divisor.
\end{abstract}

\maketitle
\setcounter{tocdepth}{1}
\tableofcontents

\section*{Introduction}

Let $\Sigma$ be a compact Riemann surface of genus $g\geq 2$, and let $G$ be a connected complex semisimple Lie group. The moduli space $\cM^d(G)$ of semistable $G$-Higgs bundles of fixed topological type admits two complementary geometric descriptions. On one hand, it contains the cotangent bundle to the stable locus of the moduli space of $G$-bundles, with its tautological holomorphic symplectic form. On the other hand, the Hitchin map exhibits it as an algebraically completely integrable system, whose generic fibres are torsors for abelian varieties and whose central fibre is the nilpotent cone. This paper studies a class of subvarieties introduced by the physicist Gaiotto which naturally interact with both descriptions.

Let
\[
\rho:G\to \Sp(\V,\omega)
\]
be a symplectic representation. After using an invariant bilinear form to identify $\lie{g}$ with $\lie{g}^*$, its moment map
\[
\mu:\V\to\lie{g}
\]
is quadratic and $G$-equivariant. Once a theta characteristic $K^{1/2}$ on $\Sigma$ is fixed, a holomorphic principal $G$-bundle $P$ and a nonzero spinor
\[
\psi\in H^0(\Sigma,P(\V)\otimes K^{1/2})
\]
determine a Higgs field
\[
\mu(\psi)\in H^0(\Sigma,\Ad(P)\otimes K).
\]
The Zariski closure in $\cM^d(G)$ of the stable Higgs bundles obtained in this way is denoted
\[
\cG^d_{(\rho,K^{1/2})}\subset \cM^d(G)
\]
and is called the \emph{Gaiotto locus} associated with $(\rho,K^{1/2})$.

This construction was introduced by Gaiotto in the study of boundary conditions and mirror symmetry, and was studied by Hitchin in rank two; it also admits derived-symplectic and loop-group formulations
\cite{gaiotto2016s,spinors,ginzburg2017gaiotto,li2017gaiotto}. From the mirror-symmetry point of view, these loci are natural candidates for the supports of BAA-branes in $\cM(G)$, whose mirrors should be BBB-branes in $\cM({}^LG)$ for the Langlands dual group ${}^LG$ \cite{kapustin2006electric}. In the case studied here the Gaiotto locus turns out to be a component of the nilpotent cone. The corresponding spectral curve is therefore non-reduced, and the usual Fourier--Mukai picture available over smooth spectral curves does not directly apply. The questions addressed in this paper are algebro-geometric: we study when Gaiotto loci are Lagrangian, and we identify them explicitly in the standard symplectic Hitchin system.

The first part of the paper treats an arbitrary symplectic representation. We introduce the moduli space
\[
\cP^d_{(\rho,K^{1/2})}
\]
of semistable $(\rho,K^{1/2})$-pairs. A point is represented by a pair $(P,\psi)$, and there is a natural assignment
\[
(P,\psi)\mapsto (P,\mu(\psi)).
\]
A semistable pair need not give a semistable Higgs bundle. What is used throughout is the converse implication proved in Proposition~\ref{prop:stability-comparison}: every semistable Higgs bundle of the form $(P,\mu(\psi))$ is represented by a semistable pair. Thus the pair moduli space provides the parameter space needed to study the stable locus whose closure is $\cG^d_{(\rho,K^{1/2})}$.

We prove that every Gaiotto locus obtained from a symplectic representation is isotropic; see Proposition~\ref{prop:isotropic}. The proof is a \v Cech-theoretic version of Hitchin's variational argument: the quadratic identity defining the moment map forces the holomorphic symplectic form to vanish on tangent vectors induced by spinor deformations. We then compare the deformation complex of a pair $(P,\psi)$ with the deformation complex of the associated Higgs bundle $(P,\mu(\psi))$. This leads to a Petri-type map
\[
d\mu_\psi:
H^0(\Sigma,P(\V)\otimes K^{1/2})
\longrightarrow
H^0(\Sigma,\Ad(P)\otimes K),
\]
whose injectivity gives a local Lagrangian criterion; see Proposition~\ref{prop:local-lagrangian}. For almost-saturated symplectic representations (see Definition~\ref{def:almost-saturated}), this yields the componentwise Lagrangian criteria of Theorems~\ref{thm:almost-saturated-lagrangian} and~\ref{thm:global-lagrangian-component}, and Corollary~\ref{cor:gaiotto-union-lagrangian}.

The second part specializes to the standard representation of $\Sp_{2n}(\C)$, with $n\geq2$. In this case
\[
\mu(\psi)=\psi\otimes\psi,
\]
so the associated Higgs field is square-zero. Hence the Gaiotto locus is contained in the nilpotent cone of the Hitchin fibration
\[
h:\cM(\Sp_{2n}(\C))\to \bigoplus_{i=1}^n H^0(\Sigma,K^{2i}).
\]
The corresponding moduli space of pairs will be called the \emph{spinor moduli space}. We give an explicit stability criterion for spinor pairs in Proposition~\ref{prop:spinor-stability}, introduce the symplectic vortex equation, and prove the associated Hitchin--Kobayashi correspondence in Theorem~\ref{thm:symplectic-vortex-HK}. This gauge-theoretic description gives a Morse function on the spinor moduli space, namely the squared $L^2$-norm of the spinor.

Since $\Sp_{2n}(\C)$ is simply connected, there is no topological index and we write
\[
\cG:=\cG_{(\rho,K^{1/2})}
\]
for the standard symplectic Gaiotto locus. The relevant fixed-point component is
\[
Z\simeq \cU(\Sp_{2n-2}(\C)),
\]
represented in Higgs moduli by bundles of the form
\[
K^{-1/2}\oplus U\oplus K^{1/2},
\]
with the unique nonzero Higgs component
\[
K^{1/2}\to K^{-1/2}\otimes K
\]
equal to the identity. Let $Z^-$ be the Bia\l ynicki--Birula downward stratum associated with $Z$.

\medskip
\noindent\textbf{Main theorem (Theorem~\ref{thm:gaiotto-component-final}).}
\emph{For $n\geq2$ and for the standard representation of $\Sp_{2n}(\C)$, the Gaiotto locus is the irreducible component of the nilpotent cone obtained from the fixed component $Z$. More precisely,}
\[
\cG=\overline{Z^-}\subset \Nilp(\Sp_{2n}(\C)).
\]
\emph{In particular, $\cG$ is irreducible.}
\medskip

The proof combines two ingredients. First, a Beauville--Laszlo-type modification removes zeros of the spinor without leaving the Zariski closure of the stable Gaiotto locus; see Lemma~\ref{lem:hecke-smoothing}. Second, for nowhere-vanishing spinors one reconstructs the underlying symplectic bundle from extension data
\[
0\subset L\subset E=L^\perp\subset V,
\qquad L=K^{-1/2},
\]
where $U=E/L$ is a symplectic bundle of rank $2n-2$. The compatibility condition on the two extension classes is exactly the condition for the reconstruction to carry a symplectic form; see Lemma~\ref{lem:symplectic-reconstruction}. This produces an irreducible reconstruction stack and a dense open family
\[
\cY\subset Z^-
\]
of dimension
\[
\dim \Sp_{2n}(\C)\, (g-1),
\]
which is the dimension of every irreducible component of the nilpotent cone. It follows that
\[
\overline{\cY}=\overline{Z^-},
\]
and the zero-removal argument then gives
\[
\cG=\overline{Z^-}.
\]

Finally, we relate this component to the generalized theta divisor
\[
\Theta=
\{V\in\cU(\Sp_{2n}(\C))\,\mid\, H^0(\Sigma,V\otimes K^{1/2})\neq0\}.
\]
On the stable cotangent chart $T^*\cU^{\mathrm{st}}$, the Gaiotto component is the closure of the conormal bundle to the open stratum
\[
\Theta_1=
\{V\in\Theta^{\mathrm{st}}\,\mid\, h^0(\Sigma,V\otimes K^{1/2})=1\};
\]
see Proposition~\ref{prop:conormal-picture}. Thus the same component has two descriptions: as a Bia\l ynicki--Birula closure inside the nilpotent cone, and as a conormal variety to the one-spinor stratum of the generalized theta divisor.

The paper is organized as follows. Section~\ref{section1} introduces symplectic representations and Gaiotto loci. Section~\ref{sec:2} studies twisted pairs, and Section~\ref{sec:3} constructs the Hitchin-type map for pairs. Section~\ref{sec:isotropy} proves isotropicity. Section~\ref{sec:lagrangian-criterion} contains the deformation-theoretic comparison and the Lagrangian criteria. Section~\ref{sec:6} specializes to the standard representation and the spinor moduli space. Section~\ref{sec:7} proves the symplectic-vortex Hitchin--Kobayashi correspondence. Section~\ref{sec:8} develops the Morse theory of the spinor moduli space. Section~\ref{sec:gaiotto-component} identifies the standard symplectic Gaiotto locus with an irreducible component of the nilpotent cone and proves the conormal interpretation.

\section*{Acknowledgements}

I am deeply grateful to Nigel Hitchin for introducing me to this problem and for the many inspiring discussions throughout my DPhil. I also thank Sergey Galkin for several interesting conversations and helpful comments related to this work. This work was completed while the author was a postdoctoral fellow at PUC-Rio, supported by CNPq through the PDJ grant 153359/2024-2.

\section{Symplectic representations and the Gaiotto locus}\label{section1}

\subfile{sections/sec1}

\section{Moduli of twisted pairs}\label{sec:2}

\subfile{sections/sec2}

\section{A Hitchin-type map for pairs}\label{sec:3}

\subfile{sections/sec3}

\section{The isotropy of the Gaiotto locus}\label{sec:isotropy}

\subfile{sections/sec4}

\section{A Lagrangian criterion}\label{sec:lagrangian-criterion}

\subfile{sections/sec5}

\section{Spinor moduli space}\label{sec:6}

\subfile{sections/sec6}

\section{The symplectic vortex equation}\label{sec:7}

\subfile{sections/sec7}

\section{Morse theory}\label{sec:8}

\subfile{sections/sec8}

\section{The Gaiotto locus and the nilpotent cone}\label{sec:gaiotto-component}

\subfile{sections/sec9}

 
\bibliographystyle{alpha}
\bibliography{Bibliography}

\end{document}

%% file: sections/sec1.tex
Let $G$ be a connected complex semisimple Lie group with Lie algebra $\lie{g}$, and let
\[
\rho : G \to \Sp (\V,\omega)
\]
be a symplectic representation on a finite-dimensional complex symplectic vector space
$(\V,\omega)$. We denote again by
\[
\rho : \lie{g} \to \lie{sp}(\V,\omega)
\]
the induced Lie algebra representation.

The associated moment map
\[
\mu_0 : \V \to \lie{g}^*
\]
is given by
\[
\inn{\mu_0(v)}{\xi}=\frac12\,\omega(\rho(\xi)v,v),
\qquad
v\in \V,\ \xi\in \lie{g}.
\]
Choosing a non-degenerate $\Ad$-invariant bilinear form $B$ on $\lie{g}$, we identify
$\lie{g}\cong \lie{g}^*$ and obtain a quadratic $G$-equivariant map
\[
\mu:\V\to \lie{g}.
\]

\begin{rmk}
Using $\omega$ to identify $\V\simeq \V^*$, one obtains
\[
\End(\V)\cong \V\otimes \V,
\]
and under this identification the symplectic Lie algebra corresponds to the symmetric square:
\[
\lie{sp}(\V,\omega)\cong \Sym^2(\V).
\]
\end{rmk}

Let $\Sigma$ be a compact Riemann surface of genus $g\geq 2$, let $K$ be its canonical bundle, and fix a square
root $K^{1/2}$. Given a holomorphic principal $G$-bundle $P$, we denote by
\[
V=P(\V)
\]
the associated symplectic vector bundle. Since $\mu$ is quadratic and $G$-equivariant, it induces a
bundle map
\[
\mu : V\otimes K^{1/2}\to \Ad(P)\otimes K.
\]
Hence every spinor
\[
\psi\in H^0(\Sigma,V\otimes K^{1/2})
\]
determines a Higgs field
\[
\mu(\psi)\in H^0(\Sigma,\Ad(P)\otimes K).
\]

Fix a topological type $d\in \pi_1(G)$ for the underlying principal $G$-bundle.

\begin{defin}
We define
\[
\cG^d_{(\rho,K^{1/2})}\subset \cM^d(G)
\]
to be the Zariski closure of the locus 
\[
\left\{
(P,\Phi)\in \cM^{d}(G)^{\mathrm{st}} \st 
\Phi=\mu(\psi)\text{ for some }0\neq \psi\in H^0(\Sigma,V\otimes K^{1/2})
\right\}.
\]
\end{defin}

\begin{rmk} \hfill
\begin{enumerate}[(a)] 
    \item The definition above is independent of the chosen representative of the stable Higgs bundle. Indeed, this follows immediately from the $G$-equivariance of the moment map.
    \item Under suitable hypotheses, the locus $\cG^d_{(\rho,K^{1/2})}$ will be shown to be Lagrangian (see Section~\ref{sec:lagrangian-criterion}).
    \item These loci arise in the physics literature in the work of Gaiotto, and we shall refer to
$\cG^d_{(\rho,K^{1/2})}$ as the \emph{Gaiotto locus} associated with the symplectic
representation $\rho$ and the theta characteristic $K^{1/2}$.
   \item The Gaiotto locus depends \emph{a priori} on the choice of theta characteristic. For the standard representation of $\Sp_{2n}(\C)$, if $K^{1/2}$ and
$K^{1/2}\otimes \alpha$ differ by a $2$-torsion line bundle
$\alpha\in \Pic^0(\Sigma)[2]$, then the corresponding loci are carried to one another
by the automorphism
\[
t_\alpha:\cM(\Sp_{2n}(\C))\to \cM(\Sp_{2n}(\C)),
\qquad
(V,\Phi)\mapsto (V\otimes \alpha,\Phi).
\]
\end{enumerate}
\end{rmk}

We record the main example for the present paper.

\begin{ex}\label{ex:standard-sp}
Let $\rho$ be the standard representation of $\Sp_{2n}(\C)$ on $\C^{2n}$, endowed with its
standard symplectic form. If we take
\[
B(\xi_1,\xi_2)=-\frac12\tr(\xi_1\xi_2)
\]
on $\lie{sp}_{2n}(\C)$, then
\[
\mu(v)=v\otimes v\in \Sym^2(\C^{2n})\cong \lie{sp}_{2n}(\C).
\]
Accordingly, for a symplectic vector bundle $(V,\omega)$ on $\Sigma$ and a spinor
$\psi\in H^0(\Sigma,V\otimes K^{1/2})$, the associated Higgs field is
\[
\Phi=\psi\otimes \psi.
\]
In particular,
\[
\Phi(v)=\omega(\psi,v)\psi,
\]
so $\Phi$ has rank at most one and satisfies $\Phi^2=0$. Hence the corresponding locus
\[
\cG_{(\rho,K^{1/2})}\subset \cM (\Sp_{2n}(\C))
\]
is contained in the nilpotent cone. The topological type has been omitted from the notation since the symplectic group is simply connected. 
\end{ex}

A second basic family is obtained from arbitrary linear representations.

\begin{ex}\label{ex:W-plus-Wdual}
Let $\tau:G\to GL(\W)$ be a finite-dimensional representation. Since $G$ is semisimple and
connected, we may assume $\tau:G\to SL(\W)$. Set
\[
\V=\W\oplus \W^*
\]
with its natural symplectic form
\[
\omega\bigl((u_1,\delta_1),(u_2,\delta_2)\bigr)
=
\inn{\delta_2}{u_1}-\inn{\delta_1}{u_2}.
\]
Then $\rho=\tau\oplus \tau^*$ is a symplectic representation, and the moment map is
characterized by
\[
B(\mu(u,\delta),\cdot)=\delta(\tau(\cdot)u).
\]

When $\W$ is the standard representation of $\SL_n(\C)$, a spinor is a pair
\[
\psi=(u,\delta),
\qquad
u\in H^0(\Sigma,W\otimes K^{1/2}),\ 
\delta\in H^0(\Sigma,W^*\otimes K^{1/2}),
\]
and the associated Higgs field is
\[
\mu(\psi)=\delta\otimes u-\frac{\inn{\delta}{u}}{n}\Id
\in H^0(\Sigma,\End_0(W)\otimes K).
\]
This example, however, behaves differently from Example~\ref{ex:standard-sp}; see
Remark~\ref{rmk:W-plus-Wdual-not-lagrangian}.
\end{ex}

\begin{rmk}\label{rmk:knop-structure}
A symplectic $G$-module is called \emph{indecomposable} if it is not isomorphic to a direct
sum of two non-zero symplectic $G$-modules. A basic structural result of Knop
\cite[Theorem 2.1]{knop1} states that every indecomposable symplectic
$G$-module is of exactly one of the following two types:
\begin{enumerate}[(i)]
\item it is irreducible as a $G$-module;
\item it is of the form
\[
\V=\W\oplus \W^*,
\]
where $\W$ is an irreducible $G$-module which does not admit a $G$-invariant symplectic
form, and $\V$ carries the natural symplectic structure
\[
\omega\bigl((u_1,\delta_1),(u_2,\delta_2)\bigr)
=
\inn{\delta_2}{u_1}-\inn{\delta_1}{u_2}.
\]
\end{enumerate}
Moreover, two symplectic $G$-modules which are isomorphic as $G$-modules are already
isomorphic as symplectic $G$-modules, and every symplectic $G$-module decomposes as a
direct sum of indecomposable symplectic $G$-modules, uniquely up to permutation.
\end{rmk}

%% file: sections/sec2.tex
Let $\rho:G\to \Sp (\V,\omega)$ be a symplectic representation, and fix a topological type $d\in \pi_1(G)$ for the underlying principal $G$-bundle.

A \emph{$(\rho,K^{1/2})$-pair} on $\Sigma$ is a pair $(P,\psi)$ consisting of a holomorphic
principal $G$-bundle $P$ of topological type $d$ and a spinor
\[
\psi\in H^0(\Sigma,P(\V)\otimes K^{1/2}).
\]
Two such pairs $(P_1,\psi_1)$ and $(P_2,\psi_2)$ are isomorphic if there is an isomorphism
of principal $G$-bundles $f:P_1\to P_2$ carrying $\psi_1$ to $\psi_2$.

To define stability, let $Q\subset G$ be a parabolic subgroup, let $\chi$ be an anti-dominant
character of $Q$, and let $\sigma\in H^0(\Sigma,P/Q)$ be a reduction of structure group.
Associated with $\chi$ is an element $s_\chi\in i\lie{h}$, where $\lie{h}$ is the Lie algebra
of a maximal compact subgroup of $G$. We then define
\[
\V^-_\chi
=
\left\{
v\in \V \st  \rho(e^{ts_\chi})v \text{ remains bounded as } t\to \infty
\right\}.
\]
This is a $Q$-invariant subspace, and hence determines a subbundle
\[
V^-_{\sigma,\chi}\subset P(\V).
\]

Similarly, one defines
\[
\V^0_\chi
=
\left\{
v\in \V \st  \rho(e^{ts_\chi})v=v \text{ for all } t\in \R
\right\},
\]
which is invariant under the Levi subgroup $L\subset Q$ and therefore determines, for every
further reduction $\sigma_L\in H^0(\Sigma,P/L)$, a subbundle
\[
V^0_{\sigma_L,\chi}\subset V^-_{\sigma,\chi}.
\]

\begin{defin}
A $(\rho,K^{1/2})$-pair $(P,\psi)$ is called:
\begin{enumerate}[(i)]
\item \emph{semistable} if for every parabolic subgroup $Q\subset G$, every reduction
$\sigma\in H^0(\Sigma,P/Q)$, and every non-trivial anti-dominant character $\chi$ of $Q$
such that
\[
\psi\in H^0(\Sigma,V^-_{\sigma,\chi}\otimes K^{1/2}),
\]
one has
\[
\deg(P)(\sigma,\chi)\geq 0;
\]

\item \emph{stable} if the inequality is always strict;

\item \emph{polystable} if it is semistable and, whenever $Q\subset G$ is a parabolic subgroup,
$\sigma\in H^0(\Sigma,P/Q)$ is a reduction, and $\chi$ is a non-trivial strictly anti-dominant
character of $Q$ such that
\[
\psi\in H^0(\Sigma,V^-_{\sigma,\chi}\otimes K^{1/2})
\qquad\text{and}\qquad
\deg(P)(\sigma,\chi)=0,
\]
there exists a further reduction $\sigma_L\in H^0(\Sigma,P/L)$ to the Levi subgroup
$L\subset Q$ such that
\[
\psi\in H^0(\Sigma,V^0_{\sigma_L,\chi}\otimes K^{1/2}).
\]
\end{enumerate}
\end{defin}

These are the standard stability conditions for twisted pairs in the sense of Schmitt
\cite{schmitt}; in the present setting they agree with the stability notion used in
\cite{GPGMiR09}, where they are related to gauge-theoretic equations via the
Hitchin--Kobayashi correspondence.

\begin{prop}
There exists a quasi-projective moduli space
\[
\cP^d_{(\rho,K^{1/2})}
\]
whose closed points correspond to $S$-equivalence classes of semistable
$(\rho,K^{1/2})$-pairs of topological type $d$.
\end{prop}

\begin{proof}
This is a special case of \cite[Theorem 2.8.1.2]{schmitt}. If
\[
\V=\V_1\oplus \cdots \oplus \V_N
\]
is the decomposition of the representation into irreducible summands, then in Schmitt's
notation our moduli space is
\[
\cM^{\chi\text{-ss}}(\rho,d,\underline{L}),
\]
with
\[
\chi=0
\qquad\text{and}\qquad
\underline{L}=(L_1,\dots,L_N),
\quad
L_i=K^{1/2}\ \text{for all }i.
\]
\end{proof}

There is a natural map from pairs to Higgs bundles:
\[
(P,\psi)\mapsto (P,\mu(\psi)).
\]
The following implication is fundamental for what follows.

\begin{prop}\label{prop:stability-comparison}
If the Higgs bundle $(P,\mu(\psi))$ is semistable, respectively stable, then the pair
$(P,\psi)$ is semistable, respectively stable.
\end{prop}

\begin{proof}
Let $Q\subset G$ be a parabolic subgroup, let $\chi$ be an anti-dominant character of $Q$,
and let $\sigma$ be a reduction of $P$ to $Q$ such that
\[
\psi\in H^0(\Sigma,V^-_{\sigma,\chi}\otimes K^{1/2}).
\]
By equivariance of the moment map,
\[
\Ad(e^{ts_\chi})\mu(v)=\mu(\rho(e^{ts_\chi})v),
\]
so boundedness of $\rho(e^{ts_\chi})v$ implies boundedness of
$\Ad(e^{ts_\chi})\mu(v)$. Therefore
\[
\mu(\psi)\in H^0(\Sigma,\Ad(P)^-_{\sigma,\chi}\otimes K).
\]
Hence every destabilizing reduction for the pair is also a destabilizing reduction for the
Higgs bundle. The conclusion follows from semistability, respectively stability, of
$(P,\mu(\psi))$.
\end{proof}

\begin{rmk}
The converse to Proposition~\ref{prop:stability-comparison} fails in general. Indeed, for a
general symplectic representation the boundedness of
\[
\Ad(e^{ts})\mu(v)
\]
does not imply the boundedness of
\[
\rho(e^{ts})v.
\]
Thus the semistability of the Higgs bundle $(P,\mu(\psi))$ need not force the semistability of
the pair $(P,\psi)$. For instance, for the trivial representation the moment map is zero, and the image of the corresponding space of pairs inside $\cM(G)$ is just the zero-Higgs locus of principal $G$-bundles on $\Sigma$.
\end{rmk}

For the standard representation of $\Sp_{2n}(\C)$ one in fact gets an equivalence.

\begin{prop}\label{prop:standard-equivalence}
Consider the case where $G=\Sp_{2n}(\C)$ with the standard representation. Let $P$ be a holomorphic principal $\Sp_{2n}(\C)$-bundle on $\Sigma$, and let
\[
\psi\in H^0(\Sigma,P(\C^{2n})\otimes K^{1/2}).
\]
Then the pair $(P,\psi)$ is semistable, respectively stable, if and only if the associated
Higgs bundle $(P,\psi\otimes \psi)$ is semistable, respectively stable.
\end{prop}

\begin{proof}
For the standard representation, the condition that $(e^{ts}v)\otimes (e^{ts}v)$ remain
bounded is equivalent to boundedness of $e^{ts}v$ itself. Thus the relevant destabilizing
conditions for pairs and for the associated Higgs bundle coincide.
\end{proof}

%% file: sections/sec3.tex
We now associate to the moduli space of twisted pairs a natural analogue of the Hitchin map.
The construction depends only on the invariant theory of the symplectic $G$-module $\V$.

Let
\[
\V=\V_1\oplus \cdots \oplus \V_N
\]
be the decomposition of $\V$ into irreducible $G$-modules. For each integer $e\geq 0$ one has
\[
\Sym^e(\V^*)
=
\bigoplus_{e_1+\cdots+e_N=e}
\Sym^{e_1}(\V_1^*)\otimes \cdots \otimes \Sym^{e_N}(\V_N^*),
\]
and hence
\[
\Sym^e(\V^*)^G
=
\bigoplus_{e_1+\cdots+e_N=e}
\left(
\Sym^{e_1}(\V_1^*)\otimes \cdots \otimes \Sym^{e_N}(\V_N^*)
\right)^G.
\]

Since $G$ is reductive, the invariant ring
\[
\C[\V]^G=\Sym(\V^*)^G
\]
is finitely generated. Choose homogeneous generators
\[
f_1,\dots,f_q\in \C[\V]^G,
\]
and suppose that
\[
f_i\in
\left(
\Sym^{e_{i,1}}(\V_1^*)\otimes \cdots \otimes \Sym^{e_{i,N}}(\V_N^*)
\right)^G.
\]
Set
\[
d_i=e_{i,1}+\cdots +e_{i,N}.
\]
Since $f_i$ is homogeneous of degree $d_i$, evaluation on a spinor
\[
\psi\in H^0(\Sigma,P(\V)\otimes K^{1/2})
\]
produces a section
\[
f_i(\psi)\in H^0\left(\Sigma,(K^{1/2})^{\otimes d_i}\right).
\]
When $d_i$ is even, this line bundle is $K^{d_i/2}$; for odd $d_i$, the chosen theta characteristic is part of the construction.
In this way one obtains a morphism
\begin{equation}\label{eq:hitchin-type-pairs}
h_{\rho}:\cP^d_{(\rho,K^{1/2})}\to \cA_{\mathrm{pair}}
:=
\bigoplus_{i=1}^q H^0\left(\Sigma,(K^{1/2})^{\otimes d_i}\right),
\qquad
(P,\psi)\mapsto \bigl(f_1(\psi),\dots,f_q(\psi)\bigr),
\end{equation}
which we shall call the \emph{Hitchin-type map} for $(\rho,K^{1/2})$-pairs.

\begin{rmk}\label{rmk:knop-little-weyl}
Knop's invariant-theoretic study of symplectic representations provides a more canonical
framework for the algebra appearing here. More precisely, for a symplectic $G$-module $\V$
there is a subspace
\[
\mathfrak a^*\subset \mathfrak t^*
\]
and a finite reflection group $W_{\V}$ acting on $\mathfrak a^*$ such that the algebraic closure
of the subring of $\C[\V]^G$ obtained by pulling back coadjoint invariants via the moment map is
\[
\C[\mathfrak a^*]^{W_{\V}}.
\]
The group $W_{\V}$ is the \emph{little Weyl group} of the symplectic representation. In this
language, the Hitchin-type map above should be viewed as the algebro-geometric avatar, over the
curve $\Sigma$, of the invariant moment-map picture for the symplectic $G$-module $\V$; see
\cite{knop-invariants}. In particular, Knop proves that $W_{\V}$ is generated by reflections and
that the corresponding invariant-moment-map quotient is equidimensional.
\end{rmk}

\begin{prop}
The morphism \eqref{eq:hitchin-type-pairs} is projective.
\end{prop}

\begin{proof}
This is a special case of \cite[Prop.~2.8.1.4]{schmitt}.
\end{proof}

The geometry of $\cP^d_{(\rho,K^{1/2})}$ depends strongly on the invariant theory of the
representation $\rho$.

\begin{cor}\label{cor:projective-standard}
For the standard representation of $\Sp_{2n}(\C)$, the moduli space
\[
\cP^d_{(\rho,K^{1/2})}
\]
is projective.
\end{cor}

\begin{proof}
The group $\Sp_{2n}(\C)$ acts transitively on $\C^{2n}\setminus \{0\}$. Hence every invariant
polynomial on $\C^{2n}$ is constant, so
\[
\C[\C^{2n}]^{\Sp_{2n}(\C)}=\C.
\]
Therefore the Hitchin-type base is a point, and the morphism
\[
h_{\rho}:\cP^d_{(\rho,K^{1/2})}\to \mathrm{pt}
\]
is simply the structure morphism. Since $h_{\rho}$ is projective,
$\cP^d_{(\rho,K^{1/2})}$ is projective.
\end{proof}

There is a natural relation between the Hitchin-type map for pairs and the ordinary Hitchin
map for Higgs bundles. Indeed, the moment map
\[
\mu:\V\to \lie{g}
\]
is $G$-equivariant, so pullback by $\mu$ defines a homomorphism
\[
\mu^*:\C[\lie{g}]^G\to \C[\V]^G.
\]
Thus every invariant polynomial on $\lie{g}$ yields an invariant polynomial on $\V$, and
the invariants detected by the ordinary Hitchin map factor through the Hitchin-type map for
pairs.

\begin{rmk}
The Hitchin-type map generally carries more information than the composition with the ordinary
Hitchin map. In particular, the projectivity or non-projectivity of the moduli space of pairs is
governed by the full invariant ring $\C[\V]^G$, not only by the subring
\[
\mu^*(\C[\lie{g}]^G)\subseteq \C[\V]^G.
\]
\end{rmk}

\begin{ex}\label{ex:nonprojective-pairs}
Consider the symplectic representation
\[
\V=\W\oplus \W^*
\]
from Example~\ref{ex:W-plus-Wdual}. The natural pairing
\[
p(u,\delta)=\inn{\delta}{u}
\]
is a nonconstant $G$-invariant polynomial on $\V$. Therefore the Hitchin-type base is positive-dimensional. A projective component cannot carry this function nontrivially.

When $\W$ is the standard representation of $\SL_n(\C)$, the moment map is
\[
\mu(u,\delta)=\delta\otimes u-\frac{\inn{\delta}{u}}{n}\Id,
\]
and one checks directly that the pullbacks of the standard generators
\[
p_i(x)=\tr(x^i), \qquad i\geq 2,
\]
are all divisible by the basic invariant $p(u,\delta)=\inn{\delta}{u}$. Thus the ordinary
Hitchin invariants coming from $\mu(\psi)$ do not exhaust the invariant theory of pairs in
this example.
\end{ex}

%% file: sections/sec4.tex
We now show that the Gaiotto locus is isotropic in $\cM^d(G)$. The argument is essentially
the one proved by Hitchin in \cite{spinors}, but we present it in a \v Cech-theoretic form
adapted to our setting.

Let $(P,\psi)$ be a $(\rho,K^{1/2})$-pair and set
\[
\Phi=\mu(\psi)\in H^0(\Sigma,\Ad(P)\otimes K).
\]
The deformation complexes of $(P,\psi)$ and $(P,\Phi)$ are
\[
C^\bullet_{\mathrm{pair}}(P,\psi) : 
\Ad(P)\xrightarrow{\ \rho(\cdot)(\psi)\ }P(\V)\otimes K^{1/2}
\]
and
\[
C^\bullet(P,\Phi) : 
\Ad(P)\xrightarrow{\ [\cdot,\Phi]\ }\Ad(P)\otimes K,
\]
respectively. In particular, the tangent space to $\cM^d(G)$ at a smooth point $(P,\Phi)$ is
naturally identified with
\[
\K^1(\Sigma,C^\bullet(P,\Phi));
\]
see, for example, \cite{bis}. Using a non-degenerate $\Ad$-invariant bilinear form
$B$ on $\lie{g}$, one obtains the natural holomorphic symplectic form on the smooth locus
of $\cM^d(G)$, namely the extension of the canonical symplectic form on the cotangent bundle
of the moduli space of principal $G$-bundles.

Consider the tensor product complex $C^\bullet(P,\Phi)\otimes C^\bullet(P,\Phi)$:
\[
\begin{aligned}
\Ad(P)\otimes \Ad(P)
&\xrightarrow{\ \beta\ }
\bigl(\Ad(P)\otimes \Ad(P)\otimes K\bigr)
\oplus
\bigl((\Ad(P)\otimes K)\otimes \Ad(P)\bigr) \\
&\xrightarrow{\ \alpha\ }
(\Ad(P)\otimes K)\otimes (\Ad(P)\otimes K),
\end{aligned}
\]
where
\[
\beta(x\otimes y)=\bigl(x\otimes [y,\Phi],[x,\Phi]\otimes y\bigr),
\qquad x,y\in \Ad(P),
\]
and
\[
\alpha(x\otimes y,\,y'\otimes x')
=
[x,\Phi]\otimes y-y'\otimes [x',\Phi],
\]
with
\[
x,x'\in \Ad(P),
\qquad
y,y'\in \Ad(P)\otimes K.
\]
There is a morphism of complexes
\[
f^\bullet :  C^\bullet(P,\Phi)\otimes C^\bullet(P,\Phi)\to K[-1]
\]
whose only non-zero component is
\[
f^1(x\otimes y,\,y'\otimes x')=B(x,y')-B(y,x').
\]

The holomorphic symplectic form is obtained as follows. First, the natural cup product in
hypercohomology yields a pairing
\[
\K^1(\Sigma,C^\bullet(P,\Phi))
\otimes
\K^1(\Sigma,C^\bullet(P,\Phi))
\longrightarrow
\K^2\bigl(\Sigma,C^\bullet(P,\Phi)\otimes C^\bullet(P,\Phi)\bigr).
\]
Composing with the map induced in hypercohomology by the morphism of complexes $f^\bullet$
gives
\[
\K^1(\Sigma,C^\bullet(P,\Phi))
\otimes
\K^1(\Sigma,C^\bullet(P,\Phi))
\longrightarrow
\K^2(\Sigma,K[-1])=H^1(\Sigma,K)\cong \C,
\]
which is precisely the holomorphic symplectic form; see \cite{bis}.

Fix an open cover $\cU=\{U_i\}$ of $\Sigma$, and write
\[
U_{ij}:=U_i\cap U_j,
\qquad
U_{ijk}:=U_i\cap U_j\cap U_k.
\]
Let
\[
v=(\{s_{ij}\},\{t_i\}),
\qquad
v'=(\{s'_{ij}\},\{t'_i\})
\]
be classes in $\K^1(\Sigma,C^\bullet(P,\Phi))$. Thus
\[
s_{ij}\in H^0(U_{ij},\Ad(P)),
\qquad
t_i\in H^0(U_i,\Ad(P)\otimes K),
\]
and similarly for $v'$, with
\begin{align*}
s_{ij}+s_{jk}&=s_{ik}
&&\text{on }U_{ijk}, \\
t_i-t_j+[s_{ij},\Phi]&=0
&&\text{on }U_{ij}.
\end{align*}
The holomorphic symplectic pairing of $v$ and $v'$ is represented by the \v Cech $1$-cocycle
\[
\bigl\{\,B(s_{ij},t'_j)-B(s'_{ij},t_i)\,\bigr\}.
\]

Suppose now that $(P,\Phi)$ lies in the smooth locus of $\cG^d_{(\rho,K^{1/2})}$, so that
\[
\Phi=\mu(\psi)
\]
for some non-zero
\[
\psi\in H^0(\Sigma,P(\V)\otimes K^{1/2}).
\]
Let
\[
v=(\{s_{ij}\},\{t_i\}),
\qquad
v'=(\{s'_{ij}\},\{t'_i\})
\]
be tangent vectors to $\cG^d_{(\rho,K^{1/2})}$ at $(P,\Phi)$. By definition, there exist local
infinitesimal deformations
\[
u_i,u'_i\in H^0(U_i,P(\V)\otimes K^{1/2})
\]
such that
\begin{equation}\label{eq:spinor-constraint}
u_i-u_j+\rho(s_{ij})\psi=0
\qquad\text{on }U_{ij},
\end{equation}
and
\begin{equation}\label{eq:spinor-constraint-prime}
u'_i-u'_j+\rho(s'_{ij})\psi=0
\qquad\text{on }U_{ij}.
\end{equation}
Moreover, the induced variation of the Higgs field is
\[
t_i=d\mu_\psi(u_i),
\]
and differentiation of
\[
B(\mu(\psi),\xi)=\frac12\,\omega(\rho(\xi)\psi,\psi)
\]
shows that
\begin{equation}\label{eq:dmu-formula}
B(t_i,\xi)=\omega\bigl(\rho(\xi)\psi,u_i\bigr)
\qquad
\text{for all }\xi\in H^0(U_i,\Ad(P)).
\end{equation}
Indeed,
\[
B(d\mu_\psi(u_i),\xi)
=
\frac12\,\omega(\rho(\xi)u_i,\psi)
+
\frac12\,\omega(\rho(\xi)\psi,u_i)
=
\omega(\rho(\xi)\psi,u_i),
\]
since $\rho(\xi)\in \lie{sp}(\V,\omega)$.

Substituting \eqref{eq:dmu-formula} into the cocycle above, we obtain
\begin{align*}
B(s_{ij},t'_j)-B(s'_{ij},t_i)
&=
\omega\bigl(\rho(s_{ij})\psi,u'_j\bigr)
-
\omega\bigl(\rho(s'_{ij})\psi,u_i\bigr) \\
&=
\omega(u_j-u_i,u'_j)-\omega(u'_j-u'_i,u_i)
\qquad
\text{by \eqref{eq:spinor-constraint} and \eqref{eq:spinor-constraint-prime}} \\
&=
\omega(u_j,u'_j)-\omega(u_i,u'_i).
\end{align*}
Thus the cocycle is a \v Cech coboundary, namely the coboundary of
\[
\{\omega(u_i,u'_i)\}.
\]
Hence the holomorphic symplectic form vanishes on tangent vectors to
$\cG^d_{(\rho,K^{1/2})}$.

The Higgs bundle moduli space also carries the canonical holomorphic one-form
\[
\theta\in H^0(\cM^d(G)_{\mathrm{sm}},T^*\cM^d(G)),
\]
defined by
\[
\theta_{(P,\Phi)}(v)=\mathrm{SD}(q(v)\otimes \Phi),
\]
where
\[
q :  \K^1(\Sigma,C^\bullet(P,\Phi))\to H^1(\Sigma,\Ad(P))
\]
is the map forgetting the Higgs-field variation, and $\mathrm{SD}$ denotes Serre duality.
It is standard that $d\theta$ is the holomorphic symplectic form; see \cite{bis,bohr}.

We now show directly that $\theta$ vanishes on $\cG^d_{(\rho,K^{1/2})}$. Let
\[
v=(\{s_{ij}\},\{t_i\})
\]
be a tangent vector to $\cG^d_{(\rho,K^{1/2})}$ at $(P,\Phi=\mu(\psi))$, and let
\[
u_i\in H^0(U_i,P(\V)\otimes K^{1/2})
\]
be a compatible infinitesimal deformation of $\psi$. Then
\begin{align*}
\theta_{(P,\mu(\psi))}(v)
&=
\bigl\{\,B(s_{ij},\mu(\psi))\,\bigr\} \\
&=
\left\{
\frac12\,\omega\bigl(\rho(s_{ij})\psi,\psi\bigr)
\right\} \\
&=
\left\{
\frac12\,\omega(u_j-u_i,\psi)
\right\}
\qquad
\text{by \eqref{eq:spinor-constraint}} \\
&=
\delta\!\left(\left\{-\frac12\,\omega(u_i,\psi)\right\}\right).
\end{align*}
Therefore
\[
\theta_{(P,\mu(\psi))}(v)=0
\]
in $H^1(\Sigma,K)$.

\begin{prop}\label{prop:isotropic}
The Gaiotto locus $\cG^d_{(\rho,K^{1/2})}\subset \cM^d(G)$ is isotropic. Moreover, the
canonical holomorphic one-form $\theta$ vanishes on the smooth locus of the Gaiotto locus.
\end{prop}

%% file: sections/sec5.tex
We now study the differential of the moment-map construction and isolate a local
criterion ensuring that the corresponding Gaiotto locus is Lagrangian.

There is a functorial assignment
\[
(P,\psi)\longmapsto (P,\mu(\psi))
\]
from pairs to Higgs bundles. Since a semistable pair need not give a semistable Higgs
bundle, this assignment should not be regarded as a morphism
$\cP^d_{(\rho,K^{1/2})}\to \cM^d(G)$ on the whole coarse moduli space. We restrict it to the open locus
\[
\cP^{d,\mathrm{H-st}}_{(\rho,K^{1/2})}
:=
\left\{
(P,\psi)\in \cP^d_{(\rho,K^{1/2})}
\st 
(P,\mu(\psi)) \text{ is stable as a Higgs bundle}
\right\},
\]
where Proposition~\ref{prop:stability-comparison} implies that the pair is stable. On this
locus the assignment descends to a morphism of coarse moduli spaces
\[
F_\rho : \cP^{d,\mathrm{H-st}}_{(\rho,K^{1/2})}\to \cM^d(G),
\qquad
(P,\psi)\mapsto (P,\mu(\psi)).
\]
At a point $(P,\psi)$ of this locus, with $\Phi=\mu(\psi)$, the differential of $F_\rho$ is induced by the natural morphism of complexes
\[
C^\bullet_{\mathrm{pair}}(P,\psi)\to C^\bullet(P,\Phi),
\]
whose degree-$0$ component is the identity on $\Ad(P)$, and whose degree-$1$ component is
\[
d\mu_\psi: P(\V)\otimes K^{1/2} \to \Ad(P)\otimes K.
\]
Indeed, equivariance of the moment map gives
\[
d\mu_\psi\bigl(\rho(\xi)\psi\bigr)=[\xi,\mu(\psi)]
\qquad
\text{for all }\xi\in \Ad(P),
\]
so the two differentials are compatible.

\subsection*{Automorphisms and simplicity}

\begin{lem}\label{lem:aut-inclusion}
For every $(\rho,K^{1/2})$-pair $(P,\psi)$, one has
\[
\Aut(P,\psi)\subseteq \Aut(P,\mu(\psi)).
\]
\end{lem}

\begin{proof}
Let $f\in \Aut(P,\psi)$. Then $f$ is an automorphism of the principal $G$-bundle $P$
and $\rho(f)\psi=\psi$. By equivariance of the moment map,
\[
\mu(\rho(f)\psi)=\Ad(f)\mu(\psi).
\]
Since $\rho(f)\psi=\psi$, it follows that
\[
\Ad(f)\mu(\psi)=\mu(\psi),
\]
so $f\in \Aut(P,\mu(\psi))$.
\end{proof}

\begin{defin}
A $(\rho,K^{1/2})$-pair $(P,\psi)$ is called \emph{simple} if
\[
\Aut(P,\psi)=Z(G)\cap \ker\rho,
\]
and a $G$-Higgs bundle $(P,\Phi)$ is called \emph{simple} if
\[
\Aut(P,\Phi)=Z(G).
\]
\end{defin}

We shall use the fact that if the associated Higgs
bundle is simple, then every automorphism of the pair is central, and one can often force it to
act trivially on the representation by requiring the spinor to have non-zero components in the
relevant irreducible summands.

To formulate this precisely, we use the structure theory of symplectic $G$-modules recalled
in Remark~\ref{rmk:knop-structure}. Following the terminology inspired by Knop's notion of a
saturated symplectic representation \cite{knop1}, we make the following definition.

\begin{defin}\label{def:almost-saturated}
Let
\[
\V=\V_1\oplus \cdots \oplus \V_N
\]
be the decomposition of $\V$ into indecomposable symplectic $G$-modules. We say that
$\rho$ is \emph{almost-saturated} if each indecomposable summand occurs with multiplicity
one and, whenever one of the summands is of the form
\[
\W\oplus \W^*,
\]
the $G$-modules $\W$ and $\W^*$ are non-isomorphic.
\end{defin}

Under this hypothesis the irreducible constituents of $\V$ are pairwise non-isomorphic. Indeed,
an indecomposable symplectic module is either irreducible or of the form $\W\oplus \W^*$, and
the second condition in Definition~\ref{def:almost-saturated} rules out an isomorphism between
the two irreducible constituents of a type-two summand.

\begin{lem}\label{lem:simplicity-criterion}
Assume that $\rho$ is almost-saturated, and write the decomposition of $P(\V)$ into
irreducible associated subbundles as
\[
P(\V)=V_1\oplus \cdots \oplus V_\ell.
\]
Let
\[
\psi=\psi_1+\cdots+\psi_\ell,
\qquad
\psi_a\in H^0(\Sigma,V_a\otimes K^{1/2}).
\]
If $(P,\mu(\psi))$ is simple and each $\psi_a$ is non-zero, then $(P,\psi)$ is simple.
\end{lem}

\begin{proof}
By Lemma~\ref{lem:aut-inclusion}, any automorphism of $(P,\psi)$ is an automorphism of the
simple Higgs bundle $(P,\mu(\psi))$, hence is given by a central element of $G$.

Since the irreducible summands of $\V$ are pairwise non-isomorphic, Schur's lemma implies
that a central element acts on each irreducible summand by a scalar. Therefore any
automorphism of $(P,\psi)$ acts on each
\[
V_a\otimes K^{1/2}
\]
by multiplication by some scalar $\lambda_a\in \C^\times$. Since the automorphism fixes
$\psi$, and since every $\psi_a$ is non-zero, one must have $\lambda_a=1$ for all $a$.
Thus the automorphism acts trivially on $P(\V)$, and therefore lies in $Z(G)\cap \ker\rho$.
The reverse inclusion is automatic, so $(P,\psi)$ is simple.
\end{proof}

\subsection*{The Petri-type map}

Motivated by the role of the Petri map in Brill--Noether theory, we call the induced map on
global sections
\[
d\mu_\psi :  H^0(\Sigma,P(\V)\otimes K^{1/2})
\to
H^0(\Sigma,\Ad(P)\otimes K)
\]
the \emph{Petri-type map} associated with $(P,\psi)$.

\begin{rmk}
For the standard representation of $\Sp_{2n}(\C)$, using
\[
\lie{sp}_{2n}(\C)\cong \Sym^2(\C^{2n}),
\]
one has
\[
d\mu_\psi(\dot\psi)=\psi\otimes \dot\psi+\dot\psi\otimes \psi.
\]
Thus the Petri-type map is the symmetric analogue of the classical Petri map.
\end{rmk}

\begin{lem}\label{lem:j-injective}
If the Petri-type map
\[
d\mu_\psi :  H^0(\Sigma,P(\V)\otimes K^{1/2})
\to
H^0(\Sigma,\Ad(P)\otimes K)
\]
is injective, then:
\begin{enumerate}[(a)]
\item the induced map
\[
j : 
\K^1(\Sigma,C^\bullet_{\mathrm{pair}}(P,\psi))
\longrightarrow
\K^1(\Sigma,C^\bullet(P,\mu(\psi)))
\]
is injective;
\item and
\[
\K^2(\Sigma,C^\bullet_{\mathrm{pair}}(P,\psi))=0.
\]
\end{enumerate}
\end{lem}

\begin{proof}
For item $(a)$, the morphism of complexes
\[
C^\bullet_{\mathrm{pair}}(P,\psi)\to C^\bullet(P,\mu(\psi))
\]
induces a commutative diagram with exact rows
\[
\begin{tikzcd}[column sep=small]
{H^0(\Sigma, \Ad(P))} \arrow[d, "\Id"'] \arrow[r] & {H^0(\Sigma, P(\V)\otimes K^{1/2})} \arrow[d, "d\mu_\psi"'] \arrow[r, "\delta_{\mathrm{pair}}"] & {\K^1(\Sigma,C^\bullet_{\mathrm{pair}}(P,\psi))} \arrow[d, "j"'] \arrow[r] & {H^1(\Sigma, \Ad(P))} \arrow[d, "\Id"'] \\
{H^0(\Sigma, \Ad(P))} \arrow[r]                   & {H^0(\Sigma, \Ad(P)\otimes K)} \arrow[r, "\delta_{\mathrm{Higgs}}"]                             & {\K^1(\Sigma,C^\bullet(P,\mu(\psi)))} \arrow[r]            & {H^1(\Sigma, \Ad(P)).}                  
\end{tikzcd}
\]
Let $\eta\in \K^1(\Sigma,C^\bullet_{\mathrm{pair}}(P,\psi))$ satisfy $j(\eta)=0$. Its image in
$H^1(\Sigma,\Ad(P))$ is zero, because the right vertical map is the identity. Hence
$\eta=\delta_{\mathrm{pair}}(s)$ for some
\[
s\in H^0(\Sigma,P(\V)\otimes K^{1/2}).
\]
Since $j(\eta)=0$, exactness of the bottom row gives an element
$\xi\in H^0(\Sigma,\Ad(P))$ such that
\[
d\mu_\psi(s)=d\mu_\psi(\rho(\xi)\psi).
\]
By injectivity of the Petri-type map, $s=\rho(\xi)\psi$, and therefore
$\eta=\delta_{\mathrm{pair}}(s)=0$. Thus $j$ is injective.

For item $(b)$, use the symplectic form on $\V$ and the invariant bilinear form $B$ on
$\lie{g}$. Up to the harmless sign coming from dualizing a two-term complex, the dual complex
of
\[
C^\bullet_{\mathrm{pair}}(P,\psi) : 
\Ad(P)\to P(\V)\otimes K^{1/2}
\]
tensored by $K$ is identified with
\[
P(\V)\otimes K^{1/2}
\xrightarrow{\ d\mu_\psi\ }
\Ad(P)\otimes K.
\]
Consequently
\[
\K^0\bigl(\Sigma,(C^\bullet_{\mathrm{pair}}(P,\psi))^\vee\otimes K\bigr)
=
\ker\!\left(
d\mu_\psi :  H^0(\Sigma,P(\V)\otimes K^{1/2})
\to H^0(\Sigma,\Ad(P)\otimes K)
\right).
\]
The kernel vanishes by injectivity of the Petri-type map. Serre duality for hypercohomology
therefore gives
\[
\K^2(\Sigma,C^\bullet_{\mathrm{pair}}(P,\psi))
\cong
\K^0\bigl(\Sigma,(C^\bullet_{\mathrm{pair}}(P,\psi))^\vee\otimes K\bigr)^*
=0.
\]
\end{proof}

\subsection*{A local Lagrangian criterion}

We first record the expected dimension of the moduli space of pairs.

\begin{lem}\label{lem:expected-dimension-pairs}
For every $(\rho,K^{1/2})$-pair $(P,\psi)$ of topological type $d$, one has
\[
\chi\bigl(C^\bullet_{\mathrm{pair}}(P,\psi)\bigr)=\dim G\,(1-g).
\]
\end{lem}

\begin{proof}
Since $G$ is semisimple, every character of $G$ is trivial. Hence, for every
finite-dimensional representation of $G$, the determinant character is trivial and the
associated vector bundle has degree zero. In particular,
\[
\deg P(\V)=0.
\]
By Riemann--Roch,
\[
\chi(\Sigma,\Ad(P))=\dim G\,(1-g)
\qquad\text{and}\qquad
\chi\bigl(\Sigma,P(\V)\otimes K^{1/2}\bigr)=0.
\]
For a two-term complex $A^\bullet = (A^0\to A^1)$, the Euler characteristic of the hypercohomology is
\[
\chi(A^\bullet)=\chi(A^0)-\chi(A^1).
\]
Applying this to
\[
C^\bullet_{\mathrm{pair}}(P,\psi):\Ad(P)\to P(\V)\otimes K^{1/2}
\]
gives
\[
\chi\bigl(C^\bullet_{\mathrm{pair}}(P,\psi)\bigr)
=
\chi(\Sigma,\Ad(P))
-
\chi\bigl(\Sigma,P(\V)\otimes K^{1/2}\bigr)
=
\dim G\,(1-g).
\]
\end{proof}

The next proposition is the basic local statement.

\begin{prop}\label{prop:local-lagrangian}
Let $(P,\psi)\in \cP^{d,\mathrm{H-st}}_{(\rho,K^{1/2})}$, and set
\[
\Phi=\mu(\psi).
\]
Assume that:
\begin{enumerate}[(i)]
\item $(P,\psi)$ is a smooth point of $\cP^d_{(\rho,K^{1/2})}$ of expected dimension
\[
\dim G\, (g-1);
\]
\item $(P,\Phi)$ is a smooth point of $\cM^d(G)$ of expected dimension
\[
2\dim G\, (g-1);
\]
\item the Petri-type map
\[
d\mu_\psi :  H^0(\Sigma,P(\V)\otimes K^{1/2})
\to
H^0(\Sigma,\Ad(P)\otimes K)
\]
is injective.
\end{enumerate}
Then
\[
dF_\rho : 
T_{(P,\psi)}\cP^{d}_{(\rho,K^{1/2})}
\to
T_{(P,\Phi)}\cM^d(G)
\]
is injective, and its image is a Lagrangian subspace.
\end{prop}

\begin{proof}
By Lemma~\ref{lem:j-injective}, the injectivity of the Petri-type map implies that the induced
map
\[
j : 
\K^1(\Sigma,C^\bullet_{\mathrm{pair}}(P,\psi))
\longrightarrow
\K^1(\Sigma,C^\bullet(P,\Phi))
\]
is injective. Since $dF_\rho$ is identified with $j$, the differential of $F_\rho$ at
$(P,\psi)$ is injective.

By Proposition~\ref{prop:isotropic}, the image of $dF_\rho$ is isotropic. Its dimension is
\[
\dim T_{(P,\psi)}\cP^d_{(\rho,K^{1/2})}
=
\dim G\, (g-1)
=
\frac12 \dim T_{(P,\Phi)}\cM^d(G),
\]
hence this isotropic subspace is necessarily Lagrangian.
\end{proof}

\begin{rmk}\label{rmk:simple-smooth-higgs}
If $(P,\mu(\psi))$ is stable and simple, then it is a smooth point of $\cM^d(G)$ of
dimension $2\dim G\, (g-1)$; see, for example, \cite{GPGMiR09}. Hence, if $(P,\psi)$ is
a smooth point of $\cP^d_{(\rho,K^{1/2})}$ of expected dimension and the Petri-type map is
injective, Proposition~\ref{prop:local-lagrangian} shows that the image of
\[
dF_\rho : 
T_{(P,\psi)}\cP^d_{(\rho,K^{1/2})}
\to
T_{(P,\mu(\psi))}\cM^d(G)
\]
is Lagrangian.
\end{rmk}

The following theorem gives a practical criterion for the Lagrangian property.

\begin{thm}\label{thm:almost-saturated-lagrangian}
Assume that $\rho$ is almost-saturated. Write
\[
P(\V)=V_1\oplus \cdots \oplus V_\ell,
\qquad
\psi=\psi_1+\cdots+\psi_\ell,
\qquad
\psi_a\in H^0(\Sigma,V_a\otimes K^{1/2}),
\]
where the $V_a$ are the associated bundles corresponding to the irreducible summands of
$\V$. Assume moreover that each $\psi_a$ is non-zero.

If the Higgs bundle $(P,\mu(\psi))$ is stable and simple, and if the Petri-type map
\[
d\mu_\psi :  H^0(\Sigma,P(\V)\otimes K^{1/2})
\to
H^0(\Sigma,\Ad(P)\otimes K)
\]
is injective, then
\[
dF_\rho : 
T_{(P,\psi)}\cP^{d}_{(\rho,K^{1/2})}
\to
T_{(P,\mu(\psi))}\cM^d(G)
\]
is injective and its image is Lagrangian.
\end{thm}

\begin{proof}
By Proposition~\ref{prop:stability-comparison}, stability of $(P,\mu(\psi))$ implies stability of
$(P,\psi)$. By Lemma~\ref{lem:simplicity-criterion}, the almost-saturated hypothesis together
with the non-vanishing of the irreducible components of $\psi$ implies that $(P,\psi)$ is
simple.

Since $(P,\mu(\psi))$ is stable and simple, it is a smooth point of $\cM^d(G)$, and
\[
T_{(P,\mu(\psi))}\cM^d(G)
\cong
\K^1(\Sigma,C^\bullet(P,\mu(\psi))).
\]
Moreover,
\[
\dim T_{(P,\mu(\psi))}\cM^d(G)=2\dim G\, (g-1);
\]
see, for example, \cite[Propositions 3.11, 3.13 and 3.18]{GPGMiR09}.

On the pair side, simplicity gives
\[
\K^0(\Sigma,C^\bullet_{\mathrm{pair}}(P,\psi))=0,
\]
and injectivity of the Petri-type map gives
\[
\K^2(\Sigma,C^\bullet_{\mathrm{pair}}(P,\psi))=0
\]
by Lemma~\ref{lem:j-injective}. By Theorem~2.3 and Remark~2.8(ii) of \cite{bis}, infinitesimal
deformations of pairs are governed by the hypercohomology of the two-term deformation
complex; their formal smoothness argument, Theorem~3.1, applies here once
$\K^2(\Sigma,C^\bullet_{\mathrm{pair}}(P,\psi))=0$. Since the stabilizer of a simple pair is the
finite central group $Z(G)\cap\ker\rho$, and this group acts trivially on the deformation
complex, the corresponding coarse moduli space is smooth at $(P,\psi)$, with
\[
T_{(P,\psi)}\cP^d_{(\rho,K^{1/2})}
\cong
\K^1(\Sigma,C^\bullet_{\mathrm{pair}}(P,\psi)).
\]

Since $\K^0(\Sigma,C^\bullet_{\mathrm{pair}}(P,\psi))=\K^2(\Sigma,C^\bullet_{\mathrm{pair}}(P,\psi))=0$,
Lemma~\ref{lem:expected-dimension-pairs} yields
\[
\dim T_{(P,\psi)}\cP^d_{(\rho,K^{1/2})}
=
\dim \K^1(\Sigma,C^\bullet_{\mathrm{pair}}(P,\psi))
=
-\chi\bigl(C^\bullet_{\mathrm{pair}}(P,\psi)\bigr)
=
\dim G\, (g-1).
\]

Finally, injectivity of the Petri-type map implies, by Lemma~\ref{lem:j-injective}, that the map
\[
j : 
\K^1(\Sigma,C^\bullet_{\mathrm{pair}}(P,\psi))
\longrightarrow
\K^1(\Sigma,C^\bullet(P,\mu(\psi)))
\]
is injective. This is precisely the differential $dF_\rho$. By Proposition~\ref{prop:isotropic},
its image is isotropic. Since its dimension is exactly
\[
\dim G\, (g-1)
=
\frac12 \dim T_{(P,\mu(\psi))}\cM^d(G),
\]
the image is Lagrangian.
\end{proof}

\begin{rmk}\label{rmk:W-plus-Wdual-not-lagrangian}
The example of $\W\oplus \W^*$ from Example~\ref{ex:W-plus-Wdual} does not fit the
local Lagrangian criterion above. Indeed, the symplectic $G$-module $\W\oplus \W^*$
carries an additional scalar $\C^\times$-action, given by
\[
t\cdot (u,\delta)=(tu,t^{-1}\delta),
\]
which preserves the symplectic form and commutes with the $G$-action. The moment map is
invariant under this action:
\[
\mu(tu,t^{-1}\delta)=\mu(u,\delta).
\]
Consequently, whenever both components of a spinor $\psi=(u,\delta)$ are non-zero, the
differential of the moment map has a non-trivial kernel, namely the direction $(u,-\delta)$.
If one of the two components vanishes, the differential still has a kernel containing the
corresponding same-side spinor deformations. Thus the injectivity hypothesis in the local
Lagrangian criterion fails identically on the locus where the spinor has a non-zero component in
$\W\oplus \W^*$.

In particular, although Proposition~\ref{prop:isotropic} still implies that the corresponding
Gaiotto locus is isotropic, the criterion proved above does not apply to show that it is
Lagrangian. This contrasts with the standard representation of $\Sp_{2n}(\C)$ considered later
in the paper.
\end{rmk}

\begin{thm}\label{thm:global-lagrangian-component}
Assume that $G$ is complex semisimple and that $\rho$ is almost-saturated. Let
\[
S\subset \cP^{d,\mathrm{H-st}}_{(\rho,K^{1/2})}
\]
be an irreducible component of the locus of pairs $(P,\psi)$ such that
\[
\psi\neq 0.
\]

Let
\[
S^{\mathrm{gen}}\subset S
\]
be a dense Zariski-open subset on which the functions
\[
(P,\psi)\mapsto h^0\bigl(\Sigma,P(\V)\otimes K^{1/2}\bigr),
\qquad
(P,\psi)\mapsto h^0\bigl(\Sigma,\Ad(P)\otimes K\bigr)
\]
are constant. Assume that the subset of $S^{\mathrm{gen}}$ consisting of those $(P,\psi)$ such
that
\begin{enumerate}[(i)]
\item if
\[
P(\V)=V_1\oplus \cdots \oplus V_\ell,
\qquad
\psi=\psi_1+\cdots+\psi_\ell
\]
is the decomposition into irreducible associated subbundles, then each $\psi_a$ is non-zero;
\item the Higgs bundle $(P,\mu(\psi))$ is simple;
\item the Petri-type map
\[
d\mu_\psi :  H^0(\Sigma,P(\V)\otimes K^{1/2})
\to
H^0(\Sigma,\Ad(P)\otimes K)
\]
is injective
\end{enumerate}
is nonempty.

Then
\[
\cG_S:=\overline{F_\rho(S)}\subset \cM^d(G)
\]
is an irreducible Lagrangian subvariety.
\end{thm}

\begin{proof}
Let
\[
U\subset S^{\mathrm{gen}}
\]
be the locus where conditions $(i)$, $(ii)$, and $(iii)$ hold. We first note that $U$ is
Zariski open in $S^{\mathrm{gen}}$.

Condition $(i)$ is open, since each $\psi_a$ is obtained from $\psi$ by projection to the
corresponding irreducible summand, and the locus $\psi_a=0$ is closed.

Condition $(ii)$ is open on the stable Higgs locus: it is the regularly stable locus. Equivalently,
in a local universal family, the relative automorphism group modulo the central subgroup
$Z(G)$ is finite over the stable locus, and the locus where this finite group scheme has trivial
fiber is open.

For condition $(iii)$, the constancy of the dimensions
\[
h^0\bigl(\Sigma,P(\V)\otimes K^{1/2}\bigr)
\qquad\text{and}\qquad
h^0\bigl(\Sigma,\Ad(P)\otimes K\bigr)
\]
on $S^{\mathrm{gen}}$ implies, after passing to an \'etale-local universal family over
$S^{\mathrm{gen}}$, that the spaces of global sections form vector bundles over
$S^{\mathrm{gen}}$, and that the Petri-type maps fit together into a morphism of vector bundles.
Hence injectivity of the Petri-type map is an open condition on $S^{\mathrm{gen}}$.

Thus $U$ is open in $S^{\mathrm{gen}}$. By assumption it is nonempty, and since
$S^{\mathrm{gen}}$ is dense in the irreducible variety $S$, it follows that $U$ is dense in $S$.

Now fix $(P,\psi)\in U$. Since $(P,\mu(\psi))$ is stable by definition of
$\cP^{d,\mathrm{H-st}}_{(\rho,K^{1/2})}$, Proposition~\ref{prop:stability-comparison} implies that
$(P,\psi)$ is stable. By almost-saturatedness and the non-vanishing of the irreducible
components of $\psi$, Lemma~\ref{lem:simplicity-criterion} shows that $(P,\psi)$ is simple.
The injectivity of the Petri-type map gives
\[
\K^2\bigl(\Sigma,C^\bullet_{\mathrm{pair}}(P,\psi)\bigr)=0
\]
by Lemma~\ref{lem:j-injective}. By the deformation theory of pairs, as used in the proof of
Theorem~\ref{thm:almost-saturated-lagrangian}, $(P,\psi)$ is a smooth point of
$\cP^d_{(\rho,K^{1/2})}$, with
\[
T_{(P,\psi)}\cP^d_{(\rho,K^{1/2})}
\cong
\K^1\bigl(\Sigma,C^\bullet_{\mathrm{pair}}(P,\psi)\bigr).
\]
Since also $(P,\mu(\psi))$ is stable and simple, it is a smooth point of $\cM^d(G)$, with
\[
T_{(P,\mu(\psi))}\cM^d(G)
\cong
\K^1\bigl(\Sigma,C^\bullet(P,\mu(\psi))\bigr).
\]

Theorem~\ref{thm:almost-saturated-lagrangian} now applies, and shows that
\[
dF_\rho : 
T_{(P,\psi)}\cP^{d,\mathrm{H-st}}_{(\rho,K^{1/2})}
\to
T_{(P,\mu(\psi))}\cM^d(G)
\]
is injective and has Lagrangian image. In particular, $F_\rho|_U$ is unramified. Moreover,
the tangent spaces of $U$ have dimension
\[
\dim G\, (g-1),
\]
so
\[
\dim U=\dim G\, (g-1).
\]

Since $U$ is dense in the irreducible variety $S$, the closure of $F_\rho(U)$ coincides with
the closure of $F_\rho(S)$, namely $\cG_S$. Because the image of an irreducible variety under a
morphism is irreducible as a constructible set, $\cG_S$ is irreducible.

The morphism
\[
F_\rho|_U :  U\to \cG_S
\]
is of finite type and unramified, hence quasi-finite. By Zariski's Main Theorem, there exists a
dense open subset
\[
V\subset \cG_S
\]
such that
\[
(F_\rho|_U)^{-1}(V)\to V
\]
is finite. Since it is also unramified, it is finite \'etale. Therefore $V$ is smooth, and for every
point of $V$ its tangent space is identified with the image of the differential of $F_\rho$ at any
point of its preimage. These tangent spaces are Lagrangian by
Theorem~\ref{thm:almost-saturated-lagrangian}. Hence the holomorphic symplectic form of
$\cM^d(G)$ vanishes on $V$.

Finally,
\[
\dim \cG_S=\dim V=\dim G\, (g-1)=\frac12\dim \cM^d(G).
\]
Since the restriction of the holomorphic symplectic form to $\cG_S^\mathrm{sm}$ is a regular
$2$-form and vanishes on the dense open subset $V\subset \cG_S^\mathrm{sm}$, it vanishes on all
of $\cG_S^\mathrm{sm}$. Thus $\cG_S$ is Lagrangian.
\end{proof}

\begin{cor}\label{cor:gaiotto-union-lagrangian}
Assume that the hypotheses of Theorem~\ref{thm:global-lagrangian-component} hold for every
irreducible component of the locus
\[
\{(P,\psi)\in \cP^{d,\mathrm{H-st}}_{(\rho,K^{1/2})}\, \mid \, \psi\neq 0\}.
\]
Then the Gaiotto locus $\cG^d_{(\rho,K^{1/2})}$ is a finite union of irreducible Lagrangian
subvarieties of $\cM^d(G)$.

In particular, if the above locus of pairs is irreducible, then $\cG^d_{(\rho,K^{1/2})}$ is itself
an irreducible Lagrangian subvariety.
\end{cor}

\begin{proof}
By definition,
\[
\cG^d_{(\rho,K^{1/2})}
=
\overline{
F_\rho\left(
\{(P,\psi)\in \cP^{d,\mathrm{H-st}}_{(\rho,K^{1/2})}\, \mid \, \psi\neq 0\}
\right)
}.
\]
Taking irreducible components upstairs and applying
Theorem~\ref{thm:global-lagrangian-component} to each of them gives the result.
\end{proof}

%% file: sections/sec6.tex
From now on we specialize to the standard representation of $\Sp_{2n}(\C)$ on $\C^{2n}$, and in Sections~\ref{sec:6}--\ref{sec:gaiotto-component} we assume $n\geq2$. The rank-two case $n=1$ is Hitchin's original case for the standard representation of $\SL_2(\C)\simeq\Sp_2(\C)$.
Since $\Sp_{2n}(\C)$ is simply connected, there is only one topological type, and we shall omit
$d$ from the notation. We set
\begin{equation}\label{eq:notation-spinor-section}
\cP:=\cP_{(\rho,K^{1/2})},
\qquad
\cG:=\cG_{(\rho,K^{1/2})}.
\end{equation}
We shall refer to $\cP$ as the \emph{spinor moduli space}. A point of $\cP$ is represented by a
pair $(V,\psi)$, where $(V,\omega)$ is a symplectic vector bundle of rank $2n$ on $\Sigma$ and
$\psi\in H^0(\Sigma,V\otimes K^{1/2})$. By Corollary~\ref{cor:projective-standard}, the spinor moduli space $\cP$ is projective.

Recall from Example~\ref{ex:standard-sp} that for the standard representation the moment map is given by
\[
\mu(v)=v\otimes v\in \Sym^2(\C^{2n})\cong \mathfrak{sp}_{2n}(\C).
\]
Thus the Higgs field associated with a pair $(V,\psi)$ is
\[
\Phi_\psi:=\psi\otimes\psi.
\]

Our first goal is to make explicit the stability condition in this case.

\begin{prop}\label{prop:spinor-stability}
Let $(V,\psi)$ be a pair, where $(V,\omega)$ is a symplectic vector bundle on $\Sigma$ and
$\psi\in H^0(\Sigma,V\otimes K^{1/2})$. Then the following hold.
\begin{enumerate}[(i)]
\item The pair $(V,\psi)$ is semistable, respectively stable, respectively polystable, if and only if
the associated $\Sp_{2n}(\C)$-Higgs bundle
\[
(V,\psi\otimes\psi)
\]
is semistable, respectively stable, respectively polystable.

\item The pair $(V,\psi)$ is semistable, respectively stable, if and only if for every non-zero
isotropic subbundle $U\subset V$ such that
\[
\psi\in H^0(\Sigma,U^\perp\otimes K^{1/2}),
\]
one has
\[
\deg U\leq 0,
\qquad\text{respectively}\qquad
\deg U<0.
\]

\item The pair $(V,\psi)$ is polystable if and only if it is semistable and satisfies the following
condition: whenever $U\subset V$ is either
\begin{enumerate}[(a)]
\item a non-zero isotropic subbundle of degree zero with
\[
\psi\in H^0(\Sigma,U^\perp\otimes K^{1/2}),
\]
or
\item a proper coisotropic subbundle of degree zero with
\[
\psi\in H^0(\Sigma,U\otimes K^{1/2}),
\]
\end{enumerate}
there exists a complementary coisotropic, respectively isotropic, subbundle $U'\subset V$ such that
\[
V=U\oplus U'
\]
and
\[
\psi\in H^0(\Sigma,U'\otimes K^{1/2}),
\qquad\text{respectively}\qquad
\psi\in H^0(\Sigma,(U')^\perp\otimes K^{1/2}).
\]
\end{enumerate}
\end{prop}

\begin{proof}
The equivalence of semistability and stability in (i) is precisely Proposition~\ref{prop:standard-equivalence}.

We next record the relevant invariance conditions for the Higgs field $\Phi_\psi=\psi\otimes\psi$.

If $U\subset V$ is isotropic, then
\[
U\text{ is }\Phi_\psi\text{-invariant}
\iff
\psi\in H^0(\Sigma,U^\perp\otimes K^{1/2}).
\]
Indeed, if $\psi\in H^0(\Sigma,U^\perp\otimes K^{1/2})$, then for every local section $u$ of $U$
one has
\[
\Phi_\psi(u)=\omega(\psi,u)\psi=0,
\]
so $U$ is $\Phi_\psi$-invariant. Conversely, assume that $U$ is $\Phi_\psi$-invariant. Then for every
local section $u$ of $U$ one has
\[
\omega(\psi,u)\psi\in U\otimes K.
\]
At a point where $\psi\neq 0$, either $\omega(\psi,u)=0$ or else $\psi$ itself lies in $U$; since $U$
is isotropic, in both cases $\psi\in U^\perp$. At points where $\psi=0$ the same conclusion is
trivial. Thus $\psi\in H^0(\Sigma,U^\perp\otimes K^{1/2})$.

If $U\subset V$ is coisotropic, then
\[
U\text{ is }\Phi_\psi\text{-invariant}
\iff
\psi\in H^0(\Sigma,U\otimes K^{1/2}).
\]
Indeed, for a symplectic Higgs field, invariance of $U$ is equivalent to invariance of $U^\perp$.
Since $U^\perp$ is isotropic, the previous equivalence gives
\[
U^\perp\text{ is }\Phi_\psi\text{-invariant}
\iff
\psi\in H^0(\Sigma,(U^\perp)^\perp\otimes K^{1/2})
=
H^0(\Sigma,U\otimes K^{1/2}).
\]

Now (ii) is exactly the usual stability criterion for $\Sp_{2n}(\C)$-Higgs bundles, specialized to
the Higgs field $\Phi_\psi=\psi\otimes\psi$; see, for example, \cite[Theorem 4.2]{GPGMiR09}.
Likewise, (iii) is the corresponding specialization of the usual polystability criterion for
$\Sp_{2n}(\C)$-Higgs bundles. The polystable case in (i) follows immediately.
\end{proof}

\begin{rmk}
When $\psi=0$, Proposition~\ref{prop:spinor-stability} reduces to the usual notions of semistability,
stability, and polystability for symplectic vector bundles.
\end{rmk}

We write
\begin{equation}\label{eq:iota-spinor}
\iota:\cP\to \cM(\Sp_{2n}(\C)),
\qquad
(V,\psi)\mapsto (V,\psi\otimes\psi).
\end{equation}

We now record the basic geometric properties of $\cP$ needed later on. The morphism \eqref{eq:iota-spinor} identifies the zero-spinor locus with the moduli space of
symplectic bundles. More precisely, let $\cU(\Sp_{2n}(\C))$ denote the moduli space of polystable
symplectic vector bundles on $\Sigma$, viewed as the locus of Higgs bundles with zero Higgs field
inside $\cM(\Sp_{2n}(\C))$.

\begin{lem}\label{lem:iota-closed-points}
The morphism $\iota$ is injective on closed points.
\end{lem}

\begin{proof}
Suppose
\[
\iota(V_1,\psi_1)\cong \iota(V_2,\psi_2),
\]
that is, there exists an isomorphism of symplectic bundles
\[
f:V_1\xrightarrow{\sim}V_2
\]
such that
\[
f\circ(\psi_1\otimes\psi_1)=(\psi_2\otimes\psi_2)\circ f.
\]
Equivalently,
\[
f(\psi_1)\otimes f(\psi_1)=\psi_2\otimes\psi_2.
\]

If $\psi_2=0$, then the right-hand side vanishes, hence $f(\psi_1)\otimes f(\psi_1)=0$, and therefore
$\psi_1=0$. In this case the pairs are clearly isomorphic.

Assume now that $\psi_2\neq 0$. On the open subset where $\psi_2$ does not vanish, the pointwise
equality
\[
u\otimes u=v\otimes v
\]
with $v\neq 0$ implies $u=\pm v$. Hence on that open subset one has
\[
f(\psi_1)=\lambda\psi_2
\]
for some holomorphic function $\lambda$ satisfying $\lambda^2=1$. Therefore $\lambda$ is locally
constant with values in $\{\pm1\}$, hence constant since $\Sigma$ is connected. By continuity the
equality extends across the zeros of the spinors, and we obtain
\[
f(\psi_1)=\pm\psi_2
\]
globally. Since $(V_2,\psi_2)$ and $(V_2,-\psi_2)$ are isomorphic via the automorphism $-\Id_{V_2}$,
the pairs $(V_1,\psi_1)$ and $(V_2,\psi_2)$ define the same point of $\cP$.
\end{proof}

Let $\Theta = \Theta_{K^{1/2}} \subset \cU(\Sp_{2n}(\C))$ be the generalized theta divisor. Its support, which we denote by the same symbol, is 
\[
\Theta=\{V\in \cU(\Sp_{2n}(\C))\,\mid\, H^0(\Sigma,V\otimes K^{1/2})\neq 0\}.
\]
Also, let 
\[
\cP^\times:=\{(V,\psi)\in \cP\,\mid\, \psi\neq 0\}
\]
be the non-zero-spinor locus.

\begin{prop}\label{prop:theta-inclusion}
The image of the zero-spinor locus $\cP\setminus \cP^\times$ under $\iota$ is precisely
$\cU(\Sp_{2n}(\C))$. Moreover,
\[
\Theta\subseteq \overline{\iota(\cP^\times)}\cap \cU(\Sp_{2n}(\C)).
\]
\end{prop}

\begin{proof}
If $\psi=0$, then $\iota(V,0)=(V,0)$, so the zero-spinor locus maps exactly to $\cU(\Sp_{2n}(\C))$.

Now let $V\in \Theta$, and choose a non-zero section $\psi\in H^0(\Sigma,V\otimes K^{1/2})$.
Then for every $t\in \C^\times$ the pair $(V,t\psi)$ lies in $\cP^\times$, and
\[
\iota(V,t\psi)=(V,t^2\psi\otimes\psi).
\]
As $t\to 0$, this converges to $(V,0)$. Hence $(V,0)\in \overline{\iota(\cP^\times)}$, and therefore
\[
\Theta\subseteq \overline{\iota(\cP^\times)}\cap \cU(\Sp_{2n}(\C)).
\]
\end{proof}

\begin{rmk}\label{rmk:theta-intersection}
By definition of the Gaiotto locus and Proposition~\ref{prop:spinor-stability},
\[
\cG=\overline{\iota(\cP^\times_{\mathrm{st}})},
\]
where $\cP^\times_{\mathrm{st}}\subset \cP^\times$ is the locus of stable non-zero-spinor pairs.
At this stage Proposition~\ref{prop:theta-inclusion} gives the inclusion
\[
\Theta\subseteq \overline{\iota(\cP^\times)}\cap \cU(\Sp_{2n}(\C)).
\]
The sharper statement involving $\cG\cap \cU(\Sp_{2n}(\C))$ will be obtained later, once the relevant
image closure has been identified.
\end{rmk}

The next lemma will be used repeatedly in the Morse-theoretic analysis.

\begin{lem}\label{lem:Lpsi-degree}
Let $(V,\psi)$ be a semistable pair, and let $L_\psi\subset V$ be the saturation of the line subsheaf generated by the image of the morphism
\[
\psi:K^{-1/2}\to V.
\]
Then
\begin{equation}\label{eq:Lpsi-bounds}
1-g\leq \deg(L_\psi)\leq 0.
\end{equation}
Moreover:
\begin{enumerate}[(i)]
\item $\deg(L_\psi)=1-g$ if and only if $\psi$ is nowhere vanishing, in which case
\[
L_\psi\cong K^{-1/2};
\]

\item if $(V,\psi)$ is polystable and $\deg(L_\psi)=0$, then $\psi=0$.
\end{enumerate}
\end{lem}

\begin{proof}
Since $L_\psi$ is a line subbundle, it is automatically isotropic. By construction,
\[
\psi\in H^0(\Sigma,L_\psi\otimes K^{1/2})
\subseteq
H^0(\Sigma,L_\psi^\perp\otimes K^{1/2}),
\]
so semistability and Proposition~\ref{prop:spinor-stability} imply $\deg(L_\psi)\leq 0$.

On the other hand, the existence of a non-zero section of $L_\psi\otimes K^{1/2}$ implies
\[
\deg(L_\psi\otimes K^{1/2})\geq 0,
\]
hence
\[
\deg(L_\psi)\geq -\deg(K^{1/2})=1-g.
\]
This proves \eqref{eq:Lpsi-bounds}.

If $\deg(L_\psi)=1-g$, then $\deg(L_\psi\otimes K^{1/2})=0$, so the non-zero section induced by
$\psi$ has no zeros. Therefore
\[
L_\psi\otimes K^{1/2}\cong \cO_\Sigma,
\]
equivalently $L_\psi\cong K^{-1/2}$. The converse is immediate.

Finally, assume that $(V,\psi)$ is polystable and $\deg(L_\psi)=0$. By the polystability criterion in
Proposition~\ref{prop:spinor-stability}, there exists a coisotropic subbundle $U'\subset V$ such that
\[
V=L_\psi\oplus U'
\]
and
\[
\psi\in H^0(\Sigma,U'\otimes K^{1/2}).
\]
But also
\[
\psi\in H^0(\Sigma,L_\psi\otimes K^{1/2}),
\]
and since $L_\psi\cap U'=0$, it follows that $\psi=0$.
\end{proof}

\begin{ex}\label{ex:maximal-spinor}
Let
\[
V=K^{-1/2}\oplus K^{1/2}\oplus U,
\]
where $U$ is a polystable symplectic vector bundle of rank $2n-2$, and let
\[
\psi=1\in H^0(\Sigma,\cO_\Sigma)\subset H^0(\Sigma,V\otimes K^{1/2}).
\]
Then $(V,\psi)$ is polystable. Indeed, in this case
\[
L_\psi=K^{-1/2}
\qquad\text{and}\qquad
L_\psi^\perp=K^{-1/2}\oplus U,
\]
and the criterion in Proposition~\ref{prop:spinor-stability} is immediate from the polystability
of $U$.
\end{ex}

The basic question for the rest of the paper is to determine which part of the nilpotent cone is cut out by the Gaiotto locus for the standard representation.

%% file: sections/sec7.tex
We now recall the gauge-theoretic equation naturally associated with the spinor moduli space.
This will be used in the next section to introduce the Morse function on $\cP$.

Choose a Hermitian metric $h_\Sigma$ on $\Sigma$, and let $d\vol\in\Omega^{1,1}(\Sigma)$ be the
associated K\"ahler form. We denote by
\[
h_0:=h_\Sigma^{-1/2}
\]
the induced Hermitian metric on $K^{1/2}$.

Let $(V,\omega)$ be a symplectic vector bundle on $\Sigma$. A reduction of structure group
from $\Sp_{2n}(\C)$ to its maximal compact subgroup $\Sp (n)$ is equivalent to a quaternionic
structure
\[
j:V\to V
\]
compatible with $\omega$. Such a reduction determines a Hermitian metric on $V$ by
\[
h(u,v)=\omega(u,jv).
\]
Conversely, a Hermitian metric on $V$ compatible with $\omega$ determines such a quaternionic
structure. We shall freely pass between these two descriptions.

Let $A$ be the Chern connection of the Hermitian bundle $(V,h)$, and let
\[
F_A\in \Omega^{1,1}(\Sigma,\End(V))
\]
be its curvature. We also denote by
\[
\Lambda:\Omega^2(\Sigma)\to \Omega^0(\Sigma)
\]
the contraction with the K\"ahler form $d\vol$.

Now let
\[
\psi\in H^0(\Sigma,V\otimes K^{1/2})
\]
be a spinor. We shall use the pointwise description of the compact moment map in the standard
representation.

Indeed, let $W$ be a complex symplectic vector space with symplectic form $\omega$ and
compatible quaternionic structure $j$, and equip the compact Lie algebra $\mathfrak{sp}(W,j)$
with the invariant bilinear form
\[
B(\xi_1,\xi_2):=-\frac12\tr(\xi_1\xi_2).
\]
Under the identification
\[
\mathfrak{sp}(W,\omega)\cong \Sym^2(W)\subset W\otimes W,
\]
the moment map for the standard action of $\Sp (W,j)$ on $W$ is given by
\[
\mu_W(v)=\frac{i}{2}\bigl(v\otimes j(v)+j(v)\otimes v\bigr).
\]
Indeed, the tensor on the right lies in $\mathfrak{sp}(W,j)$, and one checks directly that
\[
B(\mu_W(v),\xi)=\frac{i}{2}\,h(\xi(v),v),
\qquad
\xi\in \mathfrak{sp}(W,j),
\]
which is the defining moment-map identity. Applying this fibrewise to the bundle $V$, and
contracting the two $K^{1/2}$-factors by means of the metric $h_0$, we obtain the section
\[
\mu_h(\psi)=\frac{i}{2}\bigl(\psi\otimes j(\psi)+j(\psi)\otimes \psi\bigr)h_0
\in \Omega^0(\Sigma,\End(V)).
\]
Here the notation $(\psi\otimes j(\psi))h_0$ means precisely contraction of the
$K^{1/2}\otimes K^{1/2}$-factor by $h_0$.

\begin{defin}
A Hermitian metric $h$ on a symplectic vector bundle $(V,\omega)$ is said to satisfy the
\emph{symplectic vortex equation} for the spinor $\psi$ if
\begin{equation}\label{eq:symplectic-vortex}
\Lambda(F_A)+\frac{i}{2}\bigl(\psi\otimes j(\psi)+j(\psi)\otimes \psi\bigr)h_0=0.
\end{equation}
\end{defin}

The point of \eqref{eq:symplectic-vortex} is that it is precisely the Hermite--Einstein equation
for $(\rho,K^{1/2})$-pairs in the case of the standard symplectic representation. More
explicitly, the left-hand side of \eqref{eq:symplectic-vortex} is the moment map for the
natural action of the unitary gauge group on the K\"ahler manifold of pairs $(A,\psi)$
consisting of an $\Sp (n)$-connection $A$ on $V$ together with a section
\[
\psi\in \Omega^0(\Sigma,V\otimes K^{1/2})
\]
which is holomorphic with respect to the holomorphic structure induced by $A$; see, for example,
\cite{relativehk,MIR00}.

We therefore obtain the following specialization of the general Hitchin--Kobayashi
correspondence for twisted pairs.

\begin{thm}\label{thm:symplectic-vortex-HK}
Let $(V,\psi)$ be a pair, where $(V,\omega)$ is a symplectic vector bundle on $\Sigma$ and
\[
\psi\in H^0(\Sigma,V\otimes K^{1/2}).
\]
Then the following are equivalent:
\begin{enumerate}[(i)]
\item $(V,\psi)$ is polystable;

\item there exists a quaternionic structure $j$ on $V$, equivalently a Hermitian metric
$h(u,v)=\omega(u,jv)$, whose Chern connection $A$ satisfies the symplectic vortex equation
\eqref{eq:symplectic-vortex}.
\end{enumerate}
Moreover, if $(V,\psi)$ is stable, then the solution metric is unique up to unitary gauge.
\end{thm}

\begin{proof}
We apply \cite[Theorem 2.24]{GPGMiR09} to the standard representation of
\[
H^\C=\Sp_{2n}(\C),
\qquad
H=\Sp (n),
\qquad
B=\C^{2n},
\qquad
L=K^{1/2}.
\]
Since $\Sp_{2n}(\C)$ is semisimple, the central parameter $\alpha$ in the general
Hermite--Einstein equation is necessarily zero. Thus \cite[Theorem 2.24]{GPGMiR09}
identifies polystable $K^{1/2}$-twisted pairs with solutions of the corresponding
moment-map equation.

In the present case, the quadratic term in that equation is exactly the compact moment map
computed above, so the Hermite--Einstein equation becomes \eqref{eq:symplectic-vortex}.
This proves the equivalence between (i) and (ii). The uniqueness statement in the stable case
is the corresponding uniqueness statement in \cite[Theorem 2.24]{GPGMiR09}.
\end{proof}

\begin{rmk}\label{rmk:algebro-gauge}
From now on we shall pass freely between the algebro-geometric and gauge-theoretic
descriptions of the spinor moduli space. In particular, when convenient, we shall represent a
polystable point of $\cP$ by a solution $(A,\psi)$ of \eqref{eq:symplectic-vortex}, where $A$
is the Chern connection of a Hermitian metric solving the equation.
\end{rmk}

%% file: sections/sec8.tex
The spinor moduli space carries a natural $\C^\times$-action
\begin{equation}\label{eq:Cstar-spinor-action}
\lambda\cdot (V,\psi)=(V,\lambda\psi).
\end{equation}
On the gauge-theoretic side, if $(A,\psi)$ represents a polystable pair, where $A$ is the
Chern connection of a Hermitian metric solving the symplectic vortex equation, the induced
$S^1$-action is
\begin{equation}\label{eq:S1-spinor-action}
e^{i\theta}\cdot(A,\psi)=(A,e^{i\theta}\psi).
\end{equation}
As in the Higgs bundle case, this action is Hamiltonian on the smooth locus, with moment map
given by the squared $L^2$-norm of the spinor:
\begin{equation}\label{eq:spinor-morse-function}
f:\cP\to \R,
\qquad
(V,\psi)\mapsto \|\psi\|_{L^2}^2:=\int_\Sigma |\psi|^2\, d\vol .
\end{equation}

\begin{prop}\label{prop:spinor-morse-proper}
The function $f$ is proper. In particular, it attains its maximum on every connected component
of $\cP$.
\end{prop}

\begin{proof}
By Section~6, the moduli space $\cP$ is projective. Hence its associated complex analytic space
is compact. Since $f$ is continuous, it is proper.
\end{proof}

We now derive a useful bound on $f$. Fix once and for all a Hermitian metric $h_\Sigma$ on
$\Sigma$ of constant curvature $-1$, with associated K\"ahler form $d\vol$, and let
\[
h_0:=h_\Sigma^{-1/2}
\]
be the induced Hermitian metric on $K^{1/2}$.

The key point is that the symplectic vortex equation controls the degree of the line bundle
generated by the spinor.

\begin{lem}\label{lem:degree-identity-spinor}
Let $(V,\psi)\in \cP$, and assume that $\psi\neq 0$. Let $h$ be a Hermitian metric on $V$
solving the symplectic vortex equation, and let
\[
L_\psi\subset V
\]
be the saturation of the line subsheaf generated by the image of the morphism
\[
\psi:K^{-1/2}\to V.
\]
Let
\[
V=L_\psi\oplus Q
\]
be the $h$-orthogonal smooth splitting, and let
\[
\beta\in \Omega^{0,1}(\Sigma,\Hom(Q,L_\psi))
\]
be the corresponding second fundamental form. Then
\begin{equation}\label{eq:degree-identity-spinor}
\deg(L_\psi)
=
-\frac{1}{4\pi}\,\|\psi\|_{L^2}^2
-\frac{1}{2\pi}\int_\Sigma |\beta|^2\,d\vol .
\end{equation}
\end{lem}

\begin{proof}
Let $A$ be the Chern connection of $(V,h)$. With respect to the $h$-orthogonal splitting
\[
V=L_\psi\oplus Q,
\]
the connection has the form
\[
A=
\begin{pmatrix}
A_\psi & \beta \\
-\beta^* & A_Q
\end{pmatrix},
\]
where $A_\psi$ and $A_Q$ are the induced Chern connections on $L_\psi$ and $Q$,
respectively, and $\beta^*$ is the adjoint of $\beta$.

The curvature of $A$ therefore has upper-left block
\[
F_{A_\psi}-\beta\wedge\beta^*.
\]
Hence Chern--Weil theory gives
\begin{equation}\label{eq:degree-Lpsi-CW}
\deg(L_\psi)
=
\frac{i}{2\pi} \int_\Sigma F_{A_\psi}
=
\frac{1}{2\pi}\int_\Sigma \tr(\pi\, i\Lambda F_A)\, d\vol
-\frac{1}{2\pi}\int_\Sigma |\beta|^2\, d\vol,
\end{equation}
where $\pi$ denotes the orthogonal projection onto $L_\psi$.

We now compute the first term using the symplectic vortex equation
\[
\Lambda(F_A)+\frac{i}{2}\bigl(\psi\otimes j(\psi)+j(\psi)\otimes\psi\bigr)h_0=0.
\]
Under the identification of tensors with endomorphisms used throughout the paper, one has
\[
(\psi\otimes j(\psi))(\psi)=\omega(j(\psi),\psi)\,\psi,
\qquad
(j(\psi)\otimes \psi)(\psi)=\omega(\psi,\psi)\,j(\psi)=0.
\]
Since
\[
\omega(j(\psi),\psi)=-h(\psi,\psi)=-|\psi|^2,
\]
it follows that the endomorphism
\[
\bigl(\psi\otimes j(\psi)+j(\psi)\otimes\psi\bigr)h_0
\]
acts on the line bundle $L_\psi$ by multiplication by $-|\psi|^2$. Therefore
\[
\tr(\pi\, i\Lambda F_A)=-\frac12\,|\psi|^2.
\]
Substituting this into \eqref{eq:degree-Lpsi-CW} yields
\eqref{eq:degree-identity-spinor}.
\end{proof}

\begin{prop}\label{prop:spinor-norm-bound}
For every $(V,\psi)\in \cP$ one has
\begin{equation}\label{eq:spinor-norm-bound}
0\leq f(V,\psi)\leq 4\pi(g-1).
\end{equation}
\end{prop}

\begin{proof}
The lower bound is immediate. If $\psi=0$, then $f(V,\psi)=0$, so there is nothing to prove.

Assume now that $\psi\neq 0$. By Lemma~\ref{lem:Lpsi-degree},
\[
1-g\leq \deg(L_\psi)\leq 0.
\]
Combining the lower bound with Lemma~\ref{lem:degree-identity-spinor}, we obtain
\[
1-g
\leq
-\frac{1}{4\pi}\,\|\psi\|_{L^2}^2
-\frac{1}{2\pi}\int_\Sigma |\beta|^2\, d\vol
\leq
-\frac{1}{4\pi}\,\|\psi\|_{L^2}^2,
\]
and hence
\[
\|\psi\|_{L^2}^2\leq 4\pi(g-1).
\]
This proves \eqref{eq:spinor-norm-bound}.
\end{proof}

\begin{rmk}\label{rmk:equality-spinor-bound}
Assume that $\psi\neq 0$. From the proof above, equality
\[
\|\psi\|_{L^2}^2=4\pi(g-1)
\]
holds if and only if
\[
\deg(L_\psi)=1-g
\qquad\text{and}\qquad
\beta=0.
\]
By Lemma~\ref{lem:Lpsi-degree}, the first condition is equivalent to
\[
L_\psi\cong K^{-1/2}
\]
and to the fact that $\psi$ is nowhere vanishing. The second condition means that
\[
L_\psi\subset V
\]
is a holomorphic direct summand. This criterion will be used in Section~9.
\end{rmk}

We now describe the critical points of the Morse function. By a theorem of Frankel \cite{frank}, a proper moment map for a Hamiltonian circle action on a compact K\"ahler manifold is a perfect Morse--Bott function. In the form needed here, this gives the following standard statement; see also \cite{GPGMn07}.

\begin{prop}\label{prop:critical-points-spinor}
The critical points of $f$ on the smooth locus $\cP^{\mathrm{sm}}$ are precisely the fixed
points of the $S^1$-action \eqref{eq:S1-spinor-action}. Moreover, if
\[
T_{(V,\psi)}\cP^{\mathrm{sm}}=\bigoplus_{\ell\in \Z} T_\ell
\]
is the weight decomposition for the infinitesimal $S^1$-action at a fixed point, then the
$\ell$-eigenspace of the Hessian of $f$ is $T_{-\ell}$. In particular, the Morse index at a
critical point is the dimension of the positive-weight subspace of the circle action.
\end{prop}

\begin{proof}
This is the standard moment-map Morse theory for Hamiltonian circle actions on K\"ahler
manifolds; see, for example, \cite[Proposition 3.3]{GPGMn07}. The argument is exactly the same
as in the Higgs bundle case.
\end{proof}

Let $(V,\psi)\in \cP^{\mathrm{sm}}$ be a fixed point of the circle action, and assume $\psi\neq 0$.
Choose a solution of the symplectic vortex equation representing this point, and let $A$ be the
corresponding Chern connection. Then there exists a one-parameter family of $\Sp (n)$-gauge
transformations $\{g_\theta\}_{\theta\in \R}$ such that
\[
g_\theta\cdot (A,\psi)=(A,e^{i\theta}\psi).
\]
Define the associated infinitesimal gauge transformation by
\[
\Upsilon:=\frac{d g_\theta}{d\theta}\Big|_{\theta=0}.
\]
Differentiating the fixed-point condition yields
\begin{equation}\label{eq:Upsilon-covariant}
d_A\Upsilon=0,
\end{equation}
and
\begin{equation}\label{eq:Upsilon-spinor}
\Upsilon\psi=i\psi.
\end{equation}

\begin{prop}\label{prop:fixed-points-spinor}
Let $(V,\psi)\in \cP^{\mathrm{sm}}$ with $\psi\neq 0$. Then $(V,\psi)$ is fixed by the circle
action if and only if there exists a covariantly constant infinitesimal gauge transformation
$\Upsilon$ such that
\[
V=\bigoplus_{\ell>0}(V_\ell\oplus V_{-\ell})\oplus V_0,
\]
where:
\begin{enumerate}[(i)]
\item $V_\ell$ is the eigensubbundle of $\Upsilon$ with eigenvalue $i\ell$;

\item
\[
\psi\in H^0(\Sigma,V_1\otimes K^{1/2});
\]

\item the quaternionic structure $j$ determined by the vortex metric satisfies
\[
j(V_\ell)=V_{-\ell},
\qquad
V_{-\ell}\cong V_\ell^*;
\]

\item the summands $V_0$ and $V_\ell\oplus V_{-\ell}$ for $\ell>1$ are polystable
symplectic vector bundles.
\end{enumerate}
\end{prop}

\begin{proof}
Assume first that $(V,\psi)$ is fixed. By \eqref{eq:Upsilon-covariant}, the endomorphism
$\Upsilon$ is parallel, hence its eigenspaces are preserved by the connection and therefore define
global holomorphic subbundles. Since the gauge transformations $g_\theta$ form a circle action,
the eigenvalues of $\Upsilon$ are integral multiples of $i$. As $\Upsilon$ is locally valued in
$\mathfrak{sp}(n)$, the eigenvalues come in opposite pairs $\pm i\ell$, together possibly with the
eigenvalue $0$. We therefore obtain a decomposition
\[
V=\bigoplus_{\ell>0}(V_\ell\oplus V_{-\ell})\oplus V_0.
\]

Let $v_\ell\in V_\ell$ and $v_m\in V_m$. Then
\[
i\ell\,\omega(v_\ell,v_m)=\omega(\Upsilon v_\ell,v_m)
=-\omega(v_\ell,\Upsilon v_m)=-im\,\omega(v_\ell,v_m).
\]
Hence $\omega(v_\ell,v_m)=0$ unless $m=-\ell$. It follows that the symplectic form pairs
$V_\ell$ non-degenerately with $V_{-\ell}$, giving an isomorphism
\[
V_{-\ell}\cong V_\ell^*.
\]
Since the quaternionic structure $j$ is parallel and anti-linear, it sends the eigenspace of
$i\ell$ to that of $-i\ell$, so also
\[
j(V_\ell)=V_{-\ell}.
\]

The relation \eqref{eq:Upsilon-spinor} shows that $\psi$ lies in the eigensubbundle of
eigenvalue $i$, that is,
\[
\psi\in H^0(\Sigma,V_1\otimes K^{1/2}).
\]

Finally, the vortex term
\[
\bigl(\psi\otimes j(\psi)+j(\psi)\otimes\psi\bigr)h_0
\]
acts only on the summand $V_1\oplus V_{-1}$. Therefore the symplectic vortex equation reduces
to the flat equation on $V_0$ and on each $V_\ell\oplus V_{-\ell}$ with $\ell>1$. These
summands are thus polystable symplectic vector bundles.

Conversely, suppose that such a decomposition exists, and define a one-parameter family of
unitary gauge transformations by letting $g_\theta$ act as multiplication by $e^{i\ell\theta}$ on
$V_\ell$, by $e^{-i\ell\theta}$ on $V_{-\ell}$, and trivially on $V_0$. Since the decomposition is
parallel, $g_\theta$ preserves the connection $A$, and since $\psi$ lies in $V_1$, one has
\[
g_\theta\cdot (A,\psi)=(A,e^{i\theta}\psi).
\]
Thus $(V,\psi)$ is fixed.
\end{proof}

\begin{rmk}\label{rmk:zero-spinor-fixed}
Pairs of the form $(V,0)$ are, of course, fixed by the circle action. They will reappear as the
minimum locus of $f$, whereas the non-zero fixed points described in
Proposition~\ref{prop:fixed-points-spinor} govern the higher critical strata.
\end{rmk}

The next step is to isolate the fixed-point component on which the Morse function attains its
maximum and to relate its downward Morse flow to the Gaiotto locus inside the nilpotent cone.

%% file: sections/sec9.tex
Recall that a point of $\cM(\Sp_{2n}(\C))$ is represented by a triple
$(V,\omega,\Phi)$, where $V$ is a rank $2n$ vector bundle, $\omega$ is a
symplectic form, and
\[
\Phi:V\to V\otimes K
\]
is skew-symmetric with respect to $\omega$. For a spinor
\[
\psi\in H^0(\Sigma,V\otimes K^{1/2}),
\]
or equivalently a morphism $K^{-1/2}\to V$, the associated Higgs field is
\[
\Phi_\psi=\psi\otimes\psi,
\qquad
\Phi_\psi(v)=\omega(v,\psi)\psi.
\]
It satisfies $\Phi_\psi^2=0$, since $\omega(\psi,\psi)=0$. Thus, the Gaiotto locus $\cG$, introduced in \eqref{eq:notation-spinor-section}, satisfies 
\[
\cG\subset \Nilp (\Sp_{2n}(\C)):=h^{-1}(0)\subset \cM(\Sp_{2n}(\C)),
\]
where 
\[
h:\cM(\Sp_{2n}(\C))\to \bigoplus_{i=1}^n H^0(\Sigma,K^{2i}) 
\]
is the Hitchin fibration. The purpose of this section is to identify $\cG$ with an irreducible component
of the nilpotent cone of $\cM(\Sp_{2n}(\C))$.

\subsection{The maximal fixed-point component}

We begin by characterizing the points where the Morse function on $\cP$ attains its
maximum.

\begin{prop}\label{prop:maximal-fixed-locus}
Let $(V,\psi)\in \cP$. Then the following are equivalent:
\begin{enumerate}[(i)]
\item
\[
f(V,\psi)=4\pi(g-1);
\]

\item $(V,\psi)$ is isomorphic to
\[
(K^{-1/2}\oplus K^{1/2}\oplus U,1),
\]
where $U$ is a polystable symplectic vector bundle of rank $2n-2$, and
\[
1\in H^0(\Sigma,\cO_\Sigma)
\subset
H^0\!\left(\Sigma,(K^{-1/2}\oplus K^{1/2}\oplus U)\otimes K^{1/2}\right)
\]
is the canonical section of the $K^{-1/2}$-summand.
\end{enumerate}
In particular, the maximum locus of $f$ is naturally identified with
\[
\cU(\Sp_{2n-2}(\C)).
\]
Every point of this maximum locus is fixed by the circle action.
\end{prop}

\begin{proof}
By Remark~\ref{rmk:equality-spinor-bound}, equality
\[
f(V,\psi)=4\pi(g-1)
\]
holds if and only if
\[
L_\psi\cong K^{-1/2}
\qquad\text{and}\qquad
\beta=0.
\]
The first condition says that the spinor is nowhere vanishing and identifies
$K^{-1/2}$ with $L_\psi$. The second says that $L_\psi\subset V$ is a
holomorphic direct summand. Since $V$ is symplectic and $L_\psi$ is isotropic,
the symplectic dual of this line summand is $L_\psi^{-1}\simeq K^{1/2}$, and the
orthogonal quotient
\[
U:=L_\psi^\perp/L_\psi
\]
is a polystable symplectic bundle. Hence
\[
(V,\psi)\cong (K^{-1/2}\oplus K^{1/2}\oplus U,1).
\]
This proves $(i)\Rightarrow(ii)$.

Conversely, every pair of the form in $(ii)$ is polystable by
Example~\ref{ex:maximal-spinor}. Moreover, the spinor $1$ is nowhere
vanishing, so $L_\psi\cong K^{-1/2}$, and the direct-sum decomposition shows
that the second fundamental form of $L_\psi$ vanishes. Hence, again by
Remark~\ref{rmk:equality-spinor-bound},
\[
f(V,\psi)=4\pi(g-1).
\]
This proves $(ii)\Rightarrow(i)$.

The identification of the maximum locus with $\cU(\Sp_{2n-2}(\C))$ is then
immediate. The maximum locus is fixed by the circle action because $f$ is a
moment map for the circle action on the smooth locus and is constant on this
component; equivalently, for the displayed representative the automorphisms
acting with weights $1,0,-1$ on $K^{-1/2},U,K^{1/2}$ identify
$(V,\psi)$ with $(V,e^{i\theta}\psi)$.
\end{proof}

Let
\[
Z\subset \cM(\Sp_{2n}(\C))
\]
be the locus represented by Higgs bundles of the following form. The underlying
symplectic bundle is
\[
V=K^{-1/2}\oplus U\oplus K^{1/2},
\qquad
U\in \cU(\Sp_{2n-2}(\C)),
\]
where the symplectic form is the natural pairing between $K^{-1/2}$ and
$K^{1/2}$, together with the symplectic form on $U$. The Higgs field
$\Phi_0$ has only one nonzero component, namely
\[
K^{1/2}\to K^{-1/2}\otimes K=K^{1/2},
\]
and this component is the identity.

The point $(V,\omega,\Phi_0)$ is fixed by the Higgs-field scaling action. Indeed, the decomposition
\[
V=K^{-1/2}\oplus U\oplus K^{1/2}
\]
defines a one-parameter family of symplectic automorphisms $g_a$, $a\in\C^\times$, acting with weights
\[
1,\quad 0,\quad -1
\]
on the three summands respectively. Since the only nonzero component of $\Phi_0$ maps
\[
K^{1/2}\to K^{-1/2}\otimes K,
\]
it has weight $2$ with respect to this action. Hence
\[
g_a\Phi_0 g_a^{-1}=a^2\Phi_0.
\]
Therefore, for every $\lambda\in\C^\times$, choosing $a$ with $a^2=\lambda$ gives an isomorphism
\[
(V,\omega,\Phi_0)\simeq (V,\omega,\lambda\Phi_0).
\]
Thus $Z\subset \cM(\Sp_{2n}(\C))^{\C^\times}$.

The morphism
\[
\iota:\cP\to \cM(\Sp_{2n}(\C)),
\qquad
(V,\psi)\mapsto (V,\psi\otimes\psi),
\]
is equivariant for the squaring homomorphism
\[
\C^\times\to\C^\times,\qquad t\mapsto t^2,
\]
where $\C^\times$ acts on $\cP$ by scaling the spinor and on
$\cM(\Sp_{2n}(\C))$ by scaling the Higgs field. Hence $\iota$ sends fixed
points of $\cP$ to fixed points of $\cM(\Sp_{2n}(\C))$. Under $\iota$, the maximum locus of Proposition~\ref{prop:maximal-fixed-locus}
maps precisely to $Z$.

\begin{lem}
\label{lem:Z-fixed-component}
The locus $Z$ is a connected component of
\[
\cM(\Sp_{2n}(\C))^{\C^\times}.
\]
Moreover, $Z$ is naturally identified with $\cU(\Sp_{2n-2}(\C))$.
\end{lem}

\begin{proof}
We use the fixed-point description of the Higgs-field scaling action on
$G$-Higgs moduli spaces in terms of $\mathbb Z$-gradings of the Lie algebra:
a fixed point is represented by a reduction to the degree-zero group of such a
grading, with Higgs field in one nonzero graded piece; see
\cite[Proposition 3.13]{GonzalezVeryStableRegular}.

For $G=\Sp_{2n}(\C)$, the standard representation translates this into a
weight decomposition of the underlying symplectic vector bundle. The symplectic
form pairs opposite weights, and the weight-zero summand is symplectic. The
locus $Z$ is precisely the fixed locus of the Hodge type with three summands
\[
K^{-1/2},\qquad U,\qquad K^{1/2},
\]
where $U$ is symplectic of rank $2n-2$, and with unique nonzero Higgs
component
\[
K^{1/2}\to K^{-1/2}\otimes K
\]
equal to an isomorphism.

The ranks, degrees and weights of the summands are discrete invariants in a
family of fixed Higgs bundles, hence they are locally constant. Therefore this
Hodge type determines an open and closed sublocus of the fixed-point locus. For
this Hodge type, the nonzero Higgs component is an isomorphism
$K^{1/2}\to K^{1/2}$, hence, after rescaling the two line summands, it is the
identity. The remaining datum is exactly the symplectic bundle $U$. Thus this
fixed sublocus is naturally identified with $\cU(\Sp_{2n-2}(\C))$, which is
connected. Hence $Z$ is a connected component.
\end{proof}

Let
\[
Z^-:=
\left\{
(V,\omega,\Phi)\in\cM(\Sp_{2n}(\C))
\st 
\lim_{\lambda\to\infty}(V,\omega,\lambda\Phi)\in Z
\right\}.
\]
The set $Z^-$ is the Bia\l ynicki--Birula downward stratum associated with
the fixed component $Z$. Since the Hitchin map is $\C^\times$-equivariant
with strictly positive weights on the Hitchin base, any point whose
$\lambda\to\infty$ limit exists must lie in the nilpotent cone. Hence
\[
Z^-\subset \Nilp(\Sp_{2n}(\C)).
\]

We shall use the Bia\l ynicki--Birula theory in the semiprojective form relevant
to Higgs moduli. The general decomposition originates in
\cite{BialynickiBirulaActions}; for the semiprojective framework and its use
for $G$-Higgs moduli, see e.g. \cite[Section~2]{hausel_hitchin_very_stable} and the references therein.
The description of downward flows for $G$-Higgs bundles in terms of
parabolic reductions and associated graded objects is recalled in
\cite[Proposition~3.26]{GonzalezVeryStableRegular}. The fact that the global
nilpotent cone is equidimensional of dimension
\[
\dim\Sp_{2n}(\C)\, (g-1)
\]
is due to Laumon \cite{LaumonNilpotentCone}.

We will prove below that $\overline{Z^-}$ is an irreducible component of $\Nilp(\Sp_{2n}(\C))$ and we have an identification 
\[
\cG=\overline{Z^-}.
\]

\subsection{Smoothing the zero divisor of the spinor}

Let $(V,\omega,\psi)$ be a stable Gaiotto pair with $\psi\neq0$. Recall
from Lemma~\ref{lem:Lpsi-degree} that $L_\psi\subset V$ denotes the saturation of the line subsheaf generated by the image of the morphism
\[
\psi:K^{-1/2}\to V.
\]
Equivalently, $L_\psi$ is the smallest line subbundle of $V$ through which
$\psi$ factors. There is a unique effective divisor $D_\psi$ such that
\[
L_\psi\simeq K^{-1/2}(D_\psi).
\]
We call $D_\psi$ the zero divisor of $\psi$. Thus $D_\psi=0$ if and only
if $K^{-1/2}\to V$ is a subbundle inclusion.

\begin{lem}[Hecke smoothing of a spinor zero]
\label{lem:hecke-smoothing}
Let $(V,\omega,\psi)$ be a stable Gaiotto pair on $\Sigma$, and let
$p\in \supp(D_\psi)$. Then the associated Higgs bundle
\[
(V,\omega,\psi\otimes\psi)
\]
lies in the Zariski closure of the locus of stable Gaiotto Higgs bundles represented
by spinors which do not vanish at $p$.

Consequently, every stable Gaiotto Higgs bundle lies in the Zariski closure of the
stable locus represented by nowhere-vanishing spinors.
\end{lem}

\begin{proof}
Let
\[
m=\operatorname{ord}_p(D_\psi)\ge 1.
\]
Fix a formal coordinate $z$ at $p$, and write
\[
\widehat D=\operatorname{Spec}\C[[z]],\qquad
\widehat D^\times=\operatorname{Spec}\C((z)),\qquad
U=\Sigma\setminus\{p\}.
\]
Trivialize $K^{1/2}$ on $\widehat D$. Then $\psi$ may be regarded, on
$\widehat D$, as a section of $V|_{\widehat D}$. Since
$m=\operatorname{ord}_p(D_\psi)$, the section $z^{-m}\psi$ is primitive in the
free $\C[[z]]$-module $V|_{\widehat D}$. A primitive vector in a symplectic
module over the local ring $\C[[z]]$ can be completed to a symplectic basis.
Hence we may choose a symplectic frame
\[
(e_1,f_1,\ldots,e_n,f_n)
\]
of $V|_{\widehat D}$ such that
\[
\psi=z^m e_1
\]
on $\widehat D$.

Define an endomorphism $N$ with respect to this symplectic frame by
\[
N(e_1)=f_1,\qquad N(f_1)=0,\qquad
N(e_i)=N(f_i)=0\quad (i\ge 2).
\]
Then $N\in\mathfrak{sp}_{2n}(\C)$ and $N^2=0$. Hence, for the coordinate
$t$ on $\A^1$,
\[
h_t:=I+t z^{-m}N\in \Sp_{2n}(\C((z))[t]),
\qquad
h_t^{-1}=I-t z^{-m}N.
\]

Let
\[
V_U:=V|_U,\qquad V_{\widehat D}:=V|_{\widehat D},
\]
and let
\[
\varphi:V_U|_{\widehat D^\times}
   \xrightarrow{\sim}
   V_{\widehat D}|_{\widehat D^\times}
\]
be the original gluing isomorphism defining $V$. Define a new gluing isomorphism
over $\widehat D^\times\times\A^1$ by
\[
\varphi_t:=h_t\circ \varphi.
\]
By Beauville--Laszlo descent in the curve form \cite{BeauvilleLaszlo}
(equivalently, in the loop-group formulation \cite[Section 5.2.2]{SorgerLectures}),
the data
\[
p_U^*V_U,\qquad p_{\widehat D}^*V_{\widehat D},\qquad \varphi_t
\]
glue to a vector bundle
\[
\cV\to \Sigma\times\A^1,
\]
where $p_U:U\times\A^1\to U$ and
$p_{\widehat D}:\widehat D\times\A^1\to\widehat D$ are the projections. Since
each $h_t$ is symplectic, the local symplectic forms glue to a relative symplectic
form
\[
\Omega:\cV\otimes\cV\to\cO_{\Sigma\times\A^1}.
\]
At $t=0$, since $h_0=I$, we recover
\[
(\cV_0,\Omega_0)\cong (V,\omega).
\]

We now glue the spinor. On $U\times\A^1$, set
\[
\Psi_U:=p_U^*(\psi|_U).
\]
On $\widehat D\times\A^1$, using the chosen trivialization of $K^{1/2}$, set
\[
\Psi_{\widehat D}:=z^m e_1+t f_1.
\]
This is regular on $\widehat D\times\A^1$. On
$\widehat D^\times\times\A^1$, the original section satisfies
\[
\varphi(\Psi_U)=z^m e_1.
\]
Therefore
\[
\varphi_t(\Psi_U)
=
h_t\bigl(\varphi(\Psi_U)\bigr)
=
h_t(z^m e_1)
=
z^m e_1+t f_1
=
\Psi_{\widehat D}.
\]
Thus the two local sections glue to a global section
\[
\Psi\in H^0\!\bigl(\Sigma\times\A^1,\cV\otimes p_\Sigma^*K^{1/2}\bigr),
\]
where $p_\Sigma:\Sigma\times\A^1\to\Sigma$ is the projection.

Using the relative symplectic form $\Omega$, and the induced identification
\[
\Ad(\cV)\cong \Sym^2(\cV),
\]
define
\[
\Phi_\Psi:=\Psi\otimes\Psi
\in H^0\!\bigl(\Sigma\times\A^1,\Ad(\cV)\otimes p_\Sigma^*K\bigr).
\]
Equivalently, on each fibre,
\[
\Phi_t(v)=\Omega_t(\Psi_t,v)\Psi_t.
\]
Thus each fibre $(\cV_t,\Omega_t,\Phi_t)$ is a Gaiotto Higgs bundle.

For $t=0$, we recover the original Higgs bundle. For $t\neq0$,
\[
\Psi_{\widehat D}|_{z=0}=t f_1\neq0,
\]
so $\Psi_t$ does not vanish at $p$. Stability is open in algebraic families.
Since the central fibre is stable, the stable locus in the base $\A^1$ is a Zariski
open neighbourhood of $0$. Hence the image of the nonzero stable fibres has the
central fibre in its Zariski closure. This proves the first assertion.

For the final assertion, apply the preceding construction successively at the
finitely many points in the support of $D_\psi$. If zeros have already been removed at
$p_1,\ldots,p_r$, and $q$ is another point of the remaining zero divisor, then the
construction at $q$ leaves the spinor unchanged on $\Sigma\setminus\{q\}$. Hence
it remains nonzero at $p_1,\ldots,p_r$. Iterating over these points and using transitivity of Zariski closure gives the
claim.
\end{proof}

Let 
\[
\cG^{\mathrm{nv}}\subset\cG
\]
denote the locus represented by stable
Gaiotto pairs $(V,\omega,\psi)$ for which $K^{-1/2}\to V$ is a subbundle
inclusion.

\begin{prop}
\label{prop:G-closure-nv}
One has
\begin{equation}
    \cG=\overline{\cG^{\mathrm{nv}}}.
\label{eq:G-closure-nv}
\end{equation}
\end{prop}

\begin{proof}
By definition, $\cG$ is the Zariski closure of the stable Gaiotto locus.
Lemma~\ref{lem:hecke-smoothing} shows that this stable locus is contained in
$\overline{\cG^{\mathrm{nv}}}$. Since $\cG^{\mathrm{nv}}\subset \cG$ and $\cG$ is
closed, the equality follows.
\end{proof}

\subsection{Reconstruction from a nowhere-vanishing spinor}
\label{subsec:nowhere-vanishing-reconstruction}

Set
\[
L:=K^{-1/2}.
\]
Let $(V,\omega,\psi)$ be a stable Gaiotto pair with nowhere-vanishing spinor.
Then $\psi$ identifies $L$ with a line subbundle of $V$. Let
\[
\omega^\flat:V\to V^*
\]
be the isomorphism induced by the symplectic form. Dualizing the inclusion
$L\hookrightarrow V$ gives a quotient $V^*\to L^{-1}$, and hence a morphism
\[
V\xrightarrow{\ \omega^\flat\ }V^*\longrightarrow L^{-1}.
\]
Its kernel is $L^\perp$. We write
\[
E:=L^\perp.
\]
Thus a nowhere-vanishing Gaiotto pair determines a filtration
\[
0\subset L\subset E\subset V
\]
with
\[
E=L^\perp,\qquad V/E\cong L^{-1}.
\]
The quotient
\[
U:=E/L
\]
inherits a symplectic form, denoted by $\theta$.

The first extension is
\begin{equation}
    0\to L\to E\to U\to0,
\label{eq:first-extension}
\end{equation}
with class
\[
\delta\in\Ext^1(U,L)\cong H^1(\Sigma,U^*\otimes L).
\]
Using the induced isomorphism
\[
\theta^\flat:U\to U^*,
\]
we define the $\theta$-dual of $\delta$ to be the class
\[
\delta^\theta\in H^1(\Sigma,U\otimes L)
\cong
\Ext^1(L^{-1},U)
\]
obtained from $\delta$ by the identification
\[
U^*\otimes L
\xrightarrow{\ (\theta^\flat)^{-1}\otimes \operatorname{id}_L\ }
U\otimes L.
\]

The second extension is
\begin{equation}
    0\to E\to V\to L^{-1}\to0,
\label{eq:second-extension}
\end{equation}
with class
\[
\widetilde\delta\in\Ext^1(L^{-1},E)
\cong
H^1(\Sigma,E\otimes L).
\]
Composing with the quotient $E\to U$ gives a natural map
\[
\Ext^1(L^{-1},E)\to \Ext^1(L^{-1},U).
\]
The symplectic compatibility condition is that the image of
$\widetilde\delta$ under this map be $\delta^\theta$. The next lemma proves
that this condition is exactly what is needed to reconstruct the symplectic
form on $V$.

\begin{lem}[Compatible extensions and symplectic reconstruction]
\label{lem:symplectic-reconstruction}
Let $L=K^{-1/2}$, let $(U,\theta)$ be a symplectic bundle of rank
$2n-2$, and let
\[
0\to L\to E\to U\to0
\]
be an extension with class $\delta\in\Ext^1(U,L)$. Let
\[
\delta^\theta\in\Ext^1(L^{-1},U)
\]
be its $\theta$-dual. Let
\[
0\to E\to V\to L^{-1}\to0
\]
be a second extension, with class
\[
\widetilde\delta\in\Ext^1(L^{-1},E).
\]
Then $V$ carries a symplectic form $\omega$ such that $L\subset V$, $L^\perp=E$, and the induced symplectic form on $E/L\cong U$
is $\theta$, if and only if the image of $\widetilde\delta$ in
\[
\Ext^1(L^{-1},U)
\]
is $\delta^\theta$.
\end{lem}

\begin{proof}
Choose an open cover $\{W_i\}$ on which $L$, $U$, $E$, and $V$ are
trivial. Put $W_{ij}:=W_i\cap W_j$. Let
\[
l_{ij}:L_j\xrightarrow{\sim}L_i,
\qquad
u_{ij}:U_j\xrightarrow{\sim}U_i
\]
be the transition functions of $L$ and $U$, respectively. We choose the
trivializations of $U$ to be symplectic, so that
\[
u_{ij}\in \Sp_{2n-2}(\cO(W_{ij})).
\]

Choose local splittings of the first extension. With respect to the local
decomposition
\[
E_i\cong L_i\oplus U_i,
\]
the transition functions of $E$ have the form
\[
e_{ij}
=
\begin{pmatrix}
l_{ij} & d_{ij}\\
0 & u_{ij}
\end{pmatrix},
\]
where
\[
d_{ij}:U_j\to L_i
\]
is the upper-right transition block. The actual \v Cech representative of
\[
\delta\in H^1(\Sigma,U^*\otimes L)
\]
is obtained from this mixed block by writing it in one local trivialization.
Thus, in the $W_i$-trivialization, it is
\[
\delta_{ij}^{(i)}
:=
d_{ij}\circ u_{ij}^{-1}:U_i\to L_i,
\]
whereas in the $W_j$-trivialization it is
\[
\delta_{ij}^{(j)}
:=
l_{ij}^{-1}\circ d_{ij}:U_j\to L_j.
\]
These are the same section of $U^*\otimes L$ written in the two
trivializations.

We now describe the $\theta$-dual cocycle. Let
\[
d_{ij}^{\theta}:L_j^{-1}\to U_i
\]
be the mixed block characterized by
\begin{equation}
    \theta_i\bigl(d_{ij}^{\theta}(s),u_{ij}(u)\bigr)
=
\bigl\langle d_{ij}(u),\,l_{ij}^{-1}(s)\bigr\rangle_L
\label{eq:theta-dual-cocycle}
\end{equation}
for local sections $u$ of $U_j$ and $s$ of $L_j^{-1}$, where $\langle\, ,\,\rangle_L:L\otimes L^{-1}\to\cO_\Sigma$ denotes the natural
pairing. Equivalently, the actual representative of $\delta^\theta$ in the
$W_i$-trivialization is
\[
(\delta^\theta_{ij})^{(i)}
=
(\theta_i^\flat)^{-1}\circ
\bigl(\delta_{ij}^{(i)}\bigr)^\vee
=
d_{ij}^{\theta}\circ l_{ij}
:
L_i^{-1}\to U_i,
\]
where $\bigl(\delta_{ij}^{(i)}\bigr)^\vee$ denotes the transpose map. In the
$W_j$-trivialization, the same representative is
\[
(\delta^\theta_{ij})^{(j)}
=
u_{ij}^{-1}\circ d_{ij}^{\theta}
:
L_j^{-1}\to U_j.
\]
This is precisely the \v Cech realization of the identification $U^*\otimes L\cong U\otimes L$ defined by $\theta^\flat$.

Now choose local splittings of the second extension. With respect to the local
decomposition
\[
V_i\cong L_i\oplus U_i\oplus L_i^{-1},
\]
the transition functions of $V$ may be written as
\begin{equation}
    v_{ij}
=
\begin{pmatrix}
l_{ij} & d_{ij} & a_{ij}\\
0 & u_{ij} & \gamma_{ij}\\
0 & 0 & l_{ij}^{-1}
\end{pmatrix}.
\label{eq:V-transition-compatible}
\end{equation}
Here
\[
a_{ij}:L_j^{-1}\to L_i,
\qquad
\gamma_{ij}:L_j^{-1}\to U_i
\]
are the mixed blocks representing the second extension. More precisely, the
pair $(a_{ij},\gamma_{ij})$ represents
\[
\widetilde\delta\in H^1(\Sigma,E\otimes L)\cong\Ext^1(L^{-1},E),
\]
and its image in
\[
H^1(\Sigma,U\otimes L)\cong\Ext^1(L^{-1},U)
\]
is represented, in the $W_i$-trivialization, by
\[
\gamma_{ij}\circ l_{ij}:L_i^{-1}\to U_i.
\]

Assume first that the image of $\widetilde\delta$ in
$\Ext^1(L^{-1},U)$ is $\delta^\theta$. After changing the local splittings
of the second extension, we may assume that the mixed blocks satisfy
\[
\gamma_{ij}=d_{ij}^{\theta}.
\]

On each local decomposition
\[
V_i\cong L_i\oplus U_i\oplus L_i^{-1},
\]
define
\begin{equation}
    \omega_i\bigl((\ell,u,s),(\ell',u',s')\bigr)
=
\langle \ell,s'\rangle_L-\langle \ell',s\rangle_L+\theta_i(u,u').
\label{eq:local-standard-symplectic-form}
\end{equation}
We claim that the transition functions $v_{ij}$ preserve these local forms.

Indeed, let
\[
x=(\ell,u,s),
\qquad
y=(\ell',u',s')
\]
be local sections of
\[
L_j\oplus U_j\oplus L_j^{-1}.
\]
Then
\[
v_{ij}(x)
=
\bigl(
l_{ij}(\ell)+d_{ij}(u)+a_{ij}(s),
\ u_{ij}(u)+\gamma_{ij}(s),
\ l_{ij}^{-1}(s)
\bigr),
\]
and similarly for $y$. Substituting this into
\eqref{eq:local-standard-symplectic-form}, and using that $u_{ij}$ preserves
$\theta$, gives

\begin{equation}
\label{eq:symplectic-transition-difference}
\begin{aligned}
&\omega_i(v_{ij}(x),v_{ij}(y))-\omega_j(x,y) \\
={}&
\bigl\langle d_{ij}(u),l_{ij}^{-1}(s')\bigr\rangle_L
-\bigl\langle d_{ij}(u'),l_{ij}^{-1}(s)\bigr\rangle_L \\
&\quad
+\bigl\langle a_{ij}(s),l_{ij}^{-1}(s')\bigr\rangle_L
-\bigl\langle a_{ij}(s'),l_{ij}^{-1}(s)\bigr\rangle_L \\
&\quad
+\theta_i\bigl(u_{ij}(u),\gamma_{ij}(s')\bigr)
+\theta_i\bigl(\gamma_{ij}(s),u_{ij}(u')\bigr)
+\theta_i\bigl(\gamma_{ij}(s),\gamma_{ij}(s')\bigr).
\end{aligned}
\end{equation}

The two terms involving $a_{ij}$ cancel because $a_{ij}:L_j^{-1}\to L_i$
is a homomorphism between line bundles, so
\[
\bigl\langle a_{ij}(s),l_{ij}^{-1}(s')\bigr\rangle_L
=
\bigl\langle a_{ij}(s'),l_{ij}^{-1}(s)\bigr\rangle_L.
\]
The last term vanishes because the image of the line bundle $L_j^{-1}$ under
$\gamma_{ij}$ is pointwise a line in the symplectic vector bundle $U_i$,
hence isotropic.

It remains to check the mixed terms. Since $\gamma_{ij}=d_{ij}^{\theta}$,
\eqref{eq:theta-dual-cocycle} gives
\[
\theta_i\bigl(\gamma_{ij}(s'),u_{ij}(u)\bigr)
=
\bigl\langle d_{ij}(u),\,l_{ij}^{-1}(s')\bigr\rangle_L \quad \text{ and } \quad \theta_i\bigl(\gamma_{ij}(s),u_{ij}(u')\bigr)
=
\bigl\langle d_{ij}(u'),\,l_{ij}^{-1}(s)\bigr\rangle_L.
\]
The second mixed term cancels the term
\[
-\bigl\langle d_{ij}(u'),l_{ij}^{-1}(s)\bigr\rangle_L.
\]
Using skew-symmetry of $\theta_i$, the first mixed term cancels the term
\[
\bigl\langle d_{ij}(u),l_{ij}^{-1}(s')\bigr\rangle_L.
\] 
Therefore the right-hand side of
\eqref{eq:symplectic-transition-difference} is zero. Hence the local forms
$\omega_i$ glue to a global symplectic form $\omega$ on $V$.

The orthogonal complement of $L\subset V$ is
locally
\[
L_i\oplus U_i\subset L_i\oplus U_i\oplus L_i^{-1},
\]
and therefore globally
\[
L^\perp=E.
\]
The induced symplectic form on $E/L\cong U$ is $\theta$.

Conversely, suppose that $V$ carries such a symplectic form. Choose local
splittings adapted to the filtration
\[
0\subset L\subset E=L^\perp\subset V
\]
so that locally the form is written as in
\eqref{eq:local-standard-symplectic-form}. With respect to these splittings,
the transition functions have the form
\eqref{eq:V-transition-compatible} and must preserve the local standard forms.
Taking $x=(0,u,0)$ and $y=(0,0,s')$ in
\eqref{eq:symplectic-transition-difference}, we obtain
\[
\bigl\langle d_{ij}(u),l_{ij}^{-1}(s')\bigr\rangle_L
+
\theta_i\bigl(u_{ij}(u),\gamma_{ij}(s')\bigr)
=
0.
\]
Equivalently,
\[
\theta_i\bigl(\gamma_{ij}(s'),u_{ij}(u)\bigr)
=
\bigl\langle d_{ij}(u),l_{ij}^{-1}(s')\bigr\rangle_L.
\]
By the defining property \eqref{eq:theta-dual-cocycle}, this says precisely
that
\[
\gamma_{ij}=d_{ij}^{\theta}
\]
up to the change of local splittings, or equivalently up to a \v Cech coboundary.
Therefore the image of $\widetilde\delta\in\Ext^1(L^{-1},E)$ in $\Ext^1(L^{-1},U)$ is exactly $\delta^\theta$.
\end{proof}

We now record the consequence of Lemma~\ref{lem:symplectic-reconstruction} for the
extension data. Fix a symplectic bundle $(U,\theta)$ of rank $2n-2$, set
$L=K^{-1/2}$, and consider a first extension
\begin{equation}
    0\to L\to E\to U\to0
\label{eq:first-extension2}
\end{equation}
with class $\delta\in\Ext^1(U,L)$. By Lemma~\ref{lem:symplectic-reconstruction},
the compatible second extensions are precisely the lifts of
\[
\delta^\theta\in \Ext^1(L^{-1},U)\cong H^1(\Sigma,U\otimes L)
\]
under the natural map
\[
\Ext^1(L^{-1},E)\cong H^1(\Sigma,E\otimes L)
   \longrightarrow
\Ext^1(L^{-1},U)\cong H^1(\Sigma,U\otimes L).
\]

Tensoring \eqref{eq:first-extension2} by $L$, we obtain
\begin{equation}
    0\to L^2
\to E\otimes L
\to U\otimes L
\to0.
\label{eq:tensored-extension}
\end{equation}
Since $H^2(\Sigma,L^2)=0$, every class
\[
\delta^\theta\in H^1(\Sigma,U\otimes L)
\]
admits such a lift. The set of lifts is a torsor under
\[
\operatorname{coker}\!\left(
H^0(\Sigma,U\otimes L)\to H^1(\Sigma,L^2)
\right).
\]
In particular, whenever $H^0(\Sigma,U\otimes L)=0$, it is a torsor under
\[
H^1(\Sigma,L^2)=H^1(\Sigma,K^{-1}).
\]

Thus, compatible data
\[
(U,\theta,\delta,\widetilde\delta)
\]
reconstruct a symplectic bundle $(V,\omega)$ with a line subbundle
$L\subset V$. The inclusion
\[
\psi:L\hookrightarrow V
\]
is a nowhere-vanishing spinor and defines the Gaiotto Higgs field
\[
\Phi_\psi=\psi\otimes\psi.
\]

We will later restrict to the open locus where $U$ is stable of degree zero. There,
$U\otimes K^{-1/2}$ is semistable of negative slope, hence
\[
H^0(\Sigma,U\otimes K^{-1/2})=0.
\]
It follows by Riemann--Roch that the first extension spaces have constant dimension,
and the compatible second extension classes form torsors under the fixed vector space
$H^1(\Sigma,K^{-1})$. Over this open locus, the reconstruction data therefore form
an iterated affine bundle of fixed rank over the corresponding open subset of
$\cU(\Sp_{2n-2}(\C))$.

\subsection{The reconstruction stack}
\label{subsec:reconstruction-stack}

We now introduce an auxiliary reconstruction stack. It is used only to organize
the extension data and prove irreducibility; the final statements are statements
in the coarse Higgs moduli space.

Set
\[
L=K^{-1/2}.
\]
Let
\[
\mathfrak R
\]
be the algebraic stack whose $S$-points consist of a relative symplectic bundle
\[
(\cU,\Theta)
\]
of rank $2n-2$ on $\Sigma\times S$, a first extension
\[
0\to p_\Sigma^*L\to \cE\to \cU\to0,
\]
and a compatible second extension
\[
0\to \cE\to \cV\to p_\Sigma^*L^{-1}\to0,
\]
in the sense of Lemma~\ref{lem:symplectic-reconstruction}. Equivalently, the
second extension class is a lift of the class determined from the first extension
using the symplectic form $\Theta$.

The reconstruction is functorial in families and produces a relative symplectic
bundle $(\cV,\Omega)$, together with the spinor induced by
\[
p_\Sigma^*L\hookrightarrow \cE\hookrightarrow \cV.
\]
Let
\[
\mathfrak R^{\mathrm{st}}\subset \mathfrak R
\]
be the open substack on which the reconstructed Gaiotto Higgs bundle is stable.

Thus an $S$-point of $\mathfrak R^{\mathrm{st}}$ determines a family
\[
(\cV,\Omega,\Phi_\Psi)
\]
of stable $\Sp_{2n}(\C)$-Higgs bundles on $\Sigma\times S$. By the defining
property of the coarse moduli space, this family determines a morphism
\[
S\to \cM(\Sp_{2n}(\C)).
\]
Equivalently, the reconstruction defines a morphism
\[
f:\mathfrak R^{\mathrm{st}}\to \cM(\Sp_{2n}(\C)),
\]
where the moduli space is viewed as an algebraic space. On closed points,
\[
f(U,\theta,\delta,\widetilde\delta)=(V,\omega,\psi\otimes\psi).
\]

We use the notation $\mathbb V(\mathcal K)$ for the vector-bundle stack
associated with a perfect complex $\mathcal K$ of amplitude contained in
$[-1,0]$. In the cases below, this is the relative extension stack whose
geometric fibres have isomorphism classes given by the relevant $\Ext^1$-groups
and stabilizer directions given by the corresponding $\Ext^0$-groups.

\begin{lem}
\label{lem:R-irreducible}
The stack $\mathfrak R$ is irreducible.
\end{lem}

\begin{proof}
Let
\[
\mathfrak B=\operatorname{Bun}(\Sp_{2n-2}(\C)).
\]
This stack is irreducible \cite{DrinfeldSimpson,SorgerLectures}. Let
\[
\pi_{\mathfrak B}:\Sigma\times\mathfrak B\to\mathfrak B
\]
be the projection, and let $(\cU,\Theta)$ denote the universal symplectic
bundle.

We first consider the stack $\mathfrak E_1$ of first extensions
\[
0\to p_\Sigma^*L\to \cE\to \cU\to0.
\]
It is the vector-bundle stack over $\mathfrak B$ associated with the perfect
complex
\[
R\pi_{\mathfrak B*}\bigl(\cU^*\otimes p_\Sigma^*L\bigr)[1]
\simeq
R\pi_{\mathfrak B*}\bigl(\cU\otimes p_\Sigma^*L\bigr)[1],
\]
where the last identification uses $\Theta$. On geometric fibres, its
isomorphism classes are
\[
\Ext^1(U,L)\cong H^1(\Sigma,U^*\otimes L)
       \cong H^1(\Sigma,U\otimes L),
\]
and its stabilizer directions are
\[
\Ext^0(U,L)\cong H^0(\Sigma,U \otimes L).
\]
Since vector-bundle stacks are smooth-locally quotients of vector bundles by
vector bundles, a vector-bundle stack over an irreducible algebraic stack is
irreducible. Therefore $\mathfrak E_1$ is irreducible.

Let
\[
\pi_{\mathfrak E_1}:\Sigma\times\mathfrak E_1\to\mathfrak E_1
\]
be the projection. Over $\Sigma\times\mathfrak E_1$, let
\[
0\to p_\Sigma^*L\to \cE\to \cU\to0
\]
be the universal first extension. The compatible second extensions are the
lifts of the section $\delta^\Theta$ under the morphism of relative extension stacks induced by $\cE\to\cU$:
\[
\mathbb V\!\left(
R\pi_{\mathfrak E_1*}\bigl(\cE\otimes p_\Sigma^*L\bigr)[1]
\right)
\longrightarrow
\mathbb V\!\left(
R\pi_{\mathfrak E_1*}\bigl(\cU\otimes p_\Sigma^*L\bigr)[1]
\right).
\]
Tensoring the universal first extension by $p_\Sigma^*L$, we obtain
\[
0\to p_\Sigma^*L^2
\to \cE\otimes p_\Sigma^*L
\to \cU\otimes p_\Sigma^*L
\to0.
\]
Since $\Sigma$ is a curve,
\[
R^2\pi_{\mathfrak E_1*}(p_\Sigma^*L^2)=0.
\]
Thus the lifting problem is unobstructed. The stack of liftings, namely the
$2$-fibre of the above morphism over the section $\delta^\Theta$, is a
torsor under the vector-bundle stack associated with
\[
R\pi_{\mathfrak E_1*}(p_\Sigma^*L^2)[1].
\]
Equivalently, on geometric fibres, the set of liftings is affine under
$H^1(\Sigma,L^2)$, while $H^0(\Sigma,L^2)$ accounts for automorphisms.

Thus $\mathfrak R$ is obtained from the irreducible stack $\mathfrak E_1$
as a torsor under a vector-bundle stack. Hence $\mathfrak R$ is irreducible.
\end{proof}

\begin{lem}
\label{lem:Rst-irreducible}
The stack $\mathfrak R^{\mathrm{st}}$ is irreducible. Consequently, the
closure of
\[
f(\mathfrak R^{\mathrm{st}})
\]
in $\cM(\Sp_{2n}(\C))$ is irreducible.
\end{lem}

\begin{proof}
By Lemma~\ref{lem:R-irreducible}, the stack $\mathfrak R$ is irreducible.
The substack
\[
\mathfrak R^{\mathrm{st}}\subset\mathfrak R
\]
is open, and it is nonempty by Lemma~\ref{lem:Y-nonempty}. Hence
$\mathfrak R^{\mathrm{st}}$ is irreducible.

The image of an irreducible topological space under a continuous map is
irreducible, and the same is true for its closure. Therefore the closure of
$f(\mathfrak R^{\mathrm{st}})$ in $\cM(\Sp_{2n}(\C))$ is irreducible.
\end{proof}

\subsection{The dense open reconstruction locus}

Let
\[
\mathfrak Y\subset\mathfrak R^{\mathrm{st}}
\]
be the substack defined by the conditions
\begin{equation}
    U\ \text{stable},
\qquad
H^0(\Sigma,U\otimes K^{1/2})=0.
\label{eq:Y-open-conditions}
\end{equation}
Set
\[
\cY:=f(\mathfrak Y)\subset \cM(\Sp_{2n}(\C))
\]
for its image in the Higgs moduli space.

\begin{lem}
\label{lem:Y-nonempty}
The substack $\mathfrak Y$ is nonempty and Zariski open in
$\mathfrak R^{\mathrm{st}}$. Hence it is dense in
$\mathfrak R^{\mathrm{st}}$, and
\begin{equation}
    \cG^{\mathrm{nv}}\subset \overline{\cY}.
\label{eq:Gnv-in-Yclosure}
\end{equation}
\end{lem}

\begin{proof}
The condition that $U$ is stable is open, and the condition
\[
H^0(\Sigma,U\otimes K^{1/2})=0
\]
is open by upper semicontinuity. Hence $\mathfrak Y$ is open in
$\mathfrak R^{\mathrm{st}}$.

It remains to prove nonemptiness. The generalized theta divisor
\[
\Theta_{2n-2}
=
\left\{
U\in\cU(\Sp_{2n-2}(\C))
\st 
H^0(\Sigma,U\otimes K^{1/2})\neq0
\right\}
\]
is a proper divisor in the moduli space of semistable symplectic bundles
\cite{BeauvilleLaszloSorgerTheta}. Hence a general stable symplectic bundle
$U$ of rank $2n-2$ satisfies
\[
H^0(\Sigma,U\otimes K^{1/2})=0.
\]
Since $U$ is stable of degree zero and $\deg K^{-1/2}<0$, we also have
\[
H^0(\Sigma,U\otimes K^{-1/2})=0.
\]
Choose such a $U$, choose a first extension class
\[
\delta\in H^1(\Sigma,U\otimes K^{-1/2}),
\]
and choose a compatible lift of $\delta^\theta$, which exists because
$H^2(\Sigma,K^{-1})=0$.

We claim that the reconstructed spinor pair $(V,\omega,\psi)$ is stable.
By Proposition~\ref{prop:spinor-stability}, it is enough to check the spinor
stability condition.

Let $F\subset V$ be a nonzero isotropic subbundle such that
\[
\psi\in H^0(\Sigma,F^\perp\otimes K^{1/2}).
\]
Since $\psi$ identifies $L=K^{-1/2}$ with a subbundle of $V$, this condition
is equivalent to
\[
L\subset F^\perp,
\]
or equivalently
\[
F\subset L^\perp=E.
\]
Set
\[
F':=F\cap L.
\]
Then $F/F'$ injects into $U=E/L$.

If $F/F'\neq0$, let $S\subset U$ be the saturation of the image of
$F/F'$. The subbundle $S$ is isotropic in $U$. Since $U$ is stable of
degree zero,
\[
\deg S<0.
\]
As saturation can only increase degree,
\[
\deg(F/F')\leq \deg S<0.
\]
Moreover, either $F'=0$, or $F'\simeq L(-D)$ for some effective divisor
$D$. Thus $\deg F'\leq0$, and therefore
\[
\deg F=\deg F'+\deg(F/F')<0.
\]

If $F/F'=0$, then $F=F'\subset L$. Since $F\neq0$, there is an effective
divisor $D$ such that
\[
F\simeq L(-D).
\]
Hence
\[
\deg F\leq \deg L=1-g<0.
\]

Thus every nonzero isotropic subbundle satisfying the spinor stability condition
has negative degree. Hence $(V,\omega,\psi)$ is stable, and therefore the
associated Higgs bundle $(V,\omega,\psi\otimes\psi)$ is stable by
Proposition~\ref{prop:spinor-stability}. This proves that $\mathfrak Y$ is
nonempty.

Now $\mathfrak R^{\mathrm{st}}$ is a nonempty open substack of the irreducible
stack $\mathfrak R$, so it is irreducible. Since $\mathfrak Y$ is a nonempty
open substack of $\mathfrak R^{\mathrm{st}}$, it is dense in
$\mathfrak R^{\mathrm{st}}$. Continuity of $f$ gives
\[
f(\mathfrak R^{\mathrm{st}})\subset \overline{f(\mathfrak Y)}
=\overline{\cY},
\]
where the closures are taken in $\cM(\Sp_{2n}(\C))$.

By the reconstruction lemma, the image of
\[
f:\mathfrak R^{\mathrm{st}}\to \cM(\Sp_{2n}(\C))
\]
on geometric points is precisely the locus
\[
\cG^{\mathrm{nv}}\subset \cM(\Sp_{2n}(\C))
\]
of stable Gaiotto Higgs bundles represented by nowhere-vanishing spinors. Hence $\cG^{\mathrm{nv}}\subset \overline{\cY}$.
\end{proof}

\subsection{Dimension and the attracting component}

We now compute the dimension of $\cY$ and identify the Bia\l ynicki--Birula
stratum containing it. Let
\[
\cU^\circ\subset \cU(\Sp_{2n-2}(\C))
\]
be the open locus of stable symplectic bundles $U$ satisfying
\[
H^0(\Sigma,U\otimes K^{1/2})=0.
\]
On this locus, stability of $U$ and $\deg K^{-1/2}<0$ give
\[
H^0(\Sigma,U\otimes K^{-1/2})=0.
\]
Thus the first extension classes form a vector bundle with fibre
\[
H^1(\Sigma,U\otimes K^{-1/2}),
\]
and the compatible second extension classes form a torsor under the fixed vector
space
\[
H^1(\Sigma,K^{-1}),
\]
since $H^0(\Sigma,K^{-1})=0$. Hence the dimension of the reconstruction
parameters over $\cU^\circ$ is computed by the usual fibre-dimension count.

The map relevant for this count is
\[
f|_{\mathfrak Y}:\mathfrak Y\to \cM(\Sp_{2n}(\C)).
\]
It is generically finite onto its image. Indeed, for a point
\[
(V,\omega,\Phi)\in f(\mathfrak Y),
\]
the Higgs field has rank one and satisfies
\[
\operatorname{im}(\Phi)=L\otimes K,
\qquad
\ker(\Phi)=L^\perp.
\]
Therefore the filtration
\[
0\subset L\subset E\subset V
\]
is intrinsic to $(V,\Phi)$:
\[
L=\operatorname{im}(\Phi)\otimes K^{-1},
\qquad
E=\ker(\Phi).
\]
The quotient $U=E/L$, its induced symplectic form, and the two extension
classes are then recovered from this filtration. The remaining ambiguity is
finite: the equality
\[
\Phi=\psi\otimes\psi
\]
recovers the spinor morphism $K^{-1/2}\to V$, equivalently the identification
$K^{-1/2}\simeq L$, up to sign. Thus the generic fibre of
$f|_{\mathfrak Y}$ is finite.

By Riemann--Roch,
\[
h^1(\Sigma,U\otimes K^{-1/2})=4(n-1)(g-1),
\qquad
h^1(\Sigma,K^{-1})=3(g-1).
\]
Together with
\[
\dim \cU(\Sp_{2n-2}(\C))=(n-1)(2n-1)(g-1),
\]
this gives
\begin{equation}
\label{eq:Y-dimension}
\begin{aligned}
\dim\cY
&=
(n-1)(2n-1)(g-1)+4(n-1)(g-1)+3(g-1)\\
&=
n(2n+1)(g-1).
\end{aligned}
\end{equation}
Equivalently,
\[
\dim\cY=\dim\Sp_{2n}(\C)\, (g-1).
\]

\begin{lem}
\label{lem:Y-in-Zminus}
We have
\[
\cY\subset Z^-.
\]
\end{lem}

\begin{proof}
Let
\[
(V,\omega,\Phi_\psi)\in\cY.
\]
By construction, $V$ carries a filtration
\[
0\subset L\subset E=L^\perp\subset V
\]
with
\[
L=K^{-1/2},\qquad E/L\simeq U,\qquad V/E\simeq L^{-1}=K^{1/2}.
\]
Moreover,
\[
\Phi_\psi(E)=0,
\qquad
\Phi_\psi(V)\subset L\otimes K.
\]
Thus, with respect to this filtration, the only nonzero graded component of the
Higgs field is the induced map
\[
V/E=L^{-1}\to L\otimes K.
\]
Since $L=K^{-1/2}$, this is a nonzero map
\[
K^{1/2}\to K^{1/2},
\]
and, after rescaling the two line summands, it is the identity.

We now take the limit under Higgs-field scaling. Choose local splittings of the
filtration, so that locally
\[
V\simeq L\oplus U\oplus L^{-1}.
\]
For $s\in\C^\times$, let
\[
g_s=\operatorname{diag}(s,1,s^{-1})
\]
with respect to this decomposition. Since $g_s$ preserves the standard
symplectic form on
\[
L\oplus U\oplus L^{-1},
\]
conjugating the transition functions of $V$ by $g_s$ gives a family of
symplectic bundles which is isomorphic to $(V,\omega)$ for $s\neq0$ and
extends to $s=0$. Its central fibre is the graded symplectic bundle
\[
L\oplus U\oplus L^{-1}.
\]
The Higgs fields
\[
g_s\,(s^{-2}\Phi_\psi)\,g_s^{-1}
\]
also extend to $s=0$. Since $\Phi_\psi$ has only the graded component
$L^{-1}\to L\otimes K$, this component has weight $s^2$ under conjugation by
$g_s$, and the factor $s^{-2}$ gives a finite nonzero limit. The central
fibre is therefore
\[
K^{-1/2}\oplus U\oplus K^{1/2}
\]
with Higgs field given by the identity component
\[
K^{1/2}\to K^{1/2}.
\]
For $s\neq0$, this family is isomorphic to
\[
(V,\omega,s^{-2}\Phi_\psi).
\]
Setting $\lambda=s^{-2}$, we obtain
\[
\lim_{\lambda\to\infty}(V,\omega,\lambda\Phi_\psi)\in Z.
\]
Hence $(V,\omega,\Phi_\psi)\in Z^-$.
\end{proof}

\begin{prop}
\label{prop:reconstruction-family-final}
One has
\[
\overline{\cY}=\overline{Z^-}.
\]
\end{prop}

\begin{proof}
By Lemma~\ref{lem:Y-in-Zminus}, $\cY\subset Z^-$, hence
\[
\overline{\cY}\subset \overline{Z^-}.
\]
Since $\mathfrak Y$ is irreducible, its image $\cY$ is irreducible, so
$\overline{\cY}$ is irreducible. By \eqref{eq:Y-dimension}, it has dimension
\[
\dim\Sp_{2n}(\C)\, (g-1).
\]
Moreover,
\[
\overline{\cY}\subset \Nilp(\Sp_{2n}(\C)).
\]
Since the nilpotent cone is equidimensional of this dimension,
$\overline{\cY}$ is an irreducible component of the nilpotent cone.

By the semiprojective Bia\l ynicki--Birula description, the downward stratum
attached to the irreducible fixed component
\[
Z\simeq \cU(\Sp_{2n-2}(\C))
\]
is irreducible, hence $\overline{Z^-}$ is irreducible. Because
\[
\overline{\cY}\subset \overline{Z^-}\subset \Nilp(\Sp_{2n}(\C)),
\]
and $\overline{\cY}$ is already an irreducible component of the nilpotent
cone, maximality of irreducible components gives $\overline{Z^-}=\overline{\cY}$.
\end{proof}

\subsection{Identification of the Gaiotto locus}

\begin{thm}
\label{thm:gaiotto-component-final}
For $n\geq2$, the Gaiotto locus $\cG$ is the irreducible component of
$\Nilp(\Sp_{2n}(\C))$ obtained as the closure of the Bia\l ynicki--Birula
downward stratum from
\[
Z\simeq \cU(\Sp_{2n-2}(\C)).
\]
Equivalently,
\[
\cG=\overline{Z^-}.
\]
In particular, $\cG$ is irreducible.
\end{thm}

\begin{proof}
By Proposition~\ref{prop:G-closure-nv}, we have
$\cG=\overline{\cG^{\mathrm{nv}}}$. Lemma~\ref{lem:Y-nonempty} gives
$\cG^{\mathrm{nv}}\subset \overline{\cY}$, hence
$\cG\subset \overline{\cY}$. Conversely, by construction
$\cY\subset \cG^{\mathrm{nv}}\subset \cG$, and since $\cG$ is closed, this
implies $\overline{\cY}\subset \cG$. Therefore
\[
\cG=\overline{\cY}.
\]
Proposition~\ref{prop:reconstruction-family-final} gives
$\overline{\cY}=\overline{Z^-}$, so $\cG=\overline{Z^-}$. The same proposition
shows that $\overline{Z^-}$ is an irreducible component of the nilpotent cone.
Hence $\cG$ is irreducible.
\end{proof}

\subsection{The conormal interpretation}

We finish by recording the relation with the generalized theta divisor. Let
\[
\cU^{\mathrm{st}}:=\cU(\Sp_{2n}(\C))^{\mathrm{st}}
\]
be the stable symplectic-bundle locus. Recall that the generalized theta divisor is
\[
\Theta
=
\left\{
V\in\cU(\Sp_{2n}(\C))
\st 
H^0(\Sigma,V\otimes K^{1/2})\neq0
\right\}.
\]
Set
\[
\Theta^{\mathrm{st}}:=\Theta\cap \cU^{\mathrm{st}},
\]
and let
\[
\Theta_1
=
\left\{
V\in\Theta^{\mathrm{st}}
\st 
h^0(\Sigma,V\otimes K^{1/2})=1
\right\}.
\]

After twisting by the fixed theta characteristic, the local Brill--Noether
problem defining $\Theta^{\mathrm{st}}$ is the $k=1$ case of the
symplectic Brill--Noether loci. By the \'etale-local construction of
\cite[Lemma~2.3 and Proposition~2.4]{BajravaniHitching},
$\Theta^{\mathrm{st}}$ has the \'etale-local structure of a symmetric
determinantal hypersurface. In particular, $\Theta_1$ is a nonempty open
subset of the smooth locus of $\Theta^{\mathrm{st}}$.

\begin{prop}
\label{prop:conormal-picture}
The Gaiotto component $\cG$ is the closure, inside
$\cM(\Sp_{2n}(\C))$, of the conormal bundle
\[
N^*_{\Theta_1/\cU^{\mathrm{st}}}
\subset T^*\cU^{\mathrm{st}}.
\]
Moreover, when the closure is taken in the stable cotangent chart, one has
\[
\cG\cap T^*\cU^{\mathrm{st}}
=
\overline{N^*_{\Theta_1/\cU^{\mathrm{st}}}}.
\]
In particular, for the cotangent projection
\[
p:T^*\cU^{\mathrm{st}}\to \cU^{\mathrm{st}},
\]
one has
\[
\cG\cap p^{-1}(\Theta_1)
=
N^*_{\Theta_1/\cU^{\mathrm{st}}}.
\]
\end{prop}

\begin{proof}
Let $V\in\Theta_1$, and choose
\[
0\neq\psi\in H^0(\Sigma,V\otimes K^{1/2}).
\]
Since $h^0(\Sigma,V\otimes K^{1/2})=1$, the Gaiotto Higgs fields over $V$
form the line
\[
\C\cdot(\psi\otimes\psi)
\subset
H^0(\Sigma,\Ad(V)\otimes K)
\simeq T^*_V\cU^{\mathrm{st}}.
\]
By Hitchin's conormal description for symplectic representations
\cite[Proposition~1]{spinors}, this line is precisely the conormal fibre to
$\Theta_1$ at $V$. Hence
\[
\cG\cap p^{-1}(\Theta_1)
=
N^*_{\Theta_1/\cU^{\mathrm{st}}}.
\]

We now pass to the closure over the higher Brill--Noether strata. In the
symmetric determinantal local model, the conormal variety is the closure of the
rank-one symmetric conormal directions appearing over the smooth corank-one
locus. Consequently, over a point $V\in\Theta^{\mathrm{st}}$, the limiting conormal directions are obtained from the rank-one cone in
\[
\Sym^2 H^0(\Sigma,V\otimes K^{1/2}).
\]
Under the symmetrised Petri map
\[
\Sym^2 H^0(\Sigma,V\otimes K^{1/2})
\to
H^0(\Sigma,\Ad(V)\otimes K),
\]
which describes the tangent spaces to the symplectic Brill--Noether loci
\cite[Proposition~2.7]{BajravaniHitching}, this cone maps to
\[
\left\{
\psi\otimes\psi
\st 
\psi\in H^0(\Sigma,V\otimes K^{1/2})
\right\}.
\]
This is exactly the fibre of the Gaiotto construction over the stable bundle
$V$. Therefore, as closed conic subvarieties of the stable cotangent chart,
\[
\cG\cap T^*\cU^{\mathrm{st}}
=
\overline{N^*_{\Theta_1/\cU^{\mathrm{st}}}},
\]
where the closure is taken in $T^*\cU^{\mathrm{st}}$.

It remains to take the closure in the full Higgs moduli space. Since
$N^*_{\Theta_1/\cU^{\mathrm{st}}}\subset\cG$ and $\cG$ is closed, its
closure inside $\cM(\Sp_{2n}(\C))$ is contained in $\cG$. On the other hand,
$\Theta_1$ is a nonempty open subset of a divisor in $\cU^{\mathrm{st}}$, so
\[
\dim N^*_{\Theta_1/\cU^{\mathrm{st}}}
=
\dim \cU(\Sp_{2n}(\C)).
\]
Thus its closure in $\cM(\Sp_{2n}(\C))$ has dimension
\[
\dim \cU(\Sp_{2n}(\C))
=
\dim\Sp_{2n}(\C)\, (g-1),
\]
which is also the dimension of the irreducible component $\cG$. Since $\cG$
is irreducible, this closure is equal to $\cG$.
\end{proof}

\begin{rmk}
Thus the Gaiotto component has two complementary descriptions: in the nilpotent
cone it is the closure of the Bia\l ynicki--Birula stratum $Z^-$, while on the
stable cotangent chart it is the conormal variety of the generalized theta
divisor. This recovers Hitchin's conormal-bundle picture for the defining
representation of $\Sp_{2n}(\C)$, and identifies its closure with the
nilpotent-cone component $\cG$.
\end{rmk}